\documentclass[11pt]{article}
\usepackage{amsmath,amssymb,amsthm}
\usepackage{tikz-cd}
\usepackage[all]{xy}
\usepackage{authblk}
\usepackage[utf8]{inputenc}
\usepackage{fontenc}
\usepackage{mathtools, nccmath, etoolbox}
\usepackage[a4paper]{geometry}
\usepackage[hidelinks]{hyperref} 
\usepackage{graphicx}
\usepackage[sort]{natbib} % references inside \cite{ref3,ref1,ref2} are sorted
\usepackage{mathdots}
\usepackage{xcolor}
\usepackage{comment}
\theoremstyle{plain}% default
\newtheorem{theorem}{Theorem}[section]
\newtheorem{proposition}[theorem]{Proposition}

\newtheorem{corollary}[theorem]{Corollary}
\theoremstyle{definition}
\newtheorem{definition}[theorem]{Definition}
\newtheorem{example}[theorem]{Example}
\theoremstyle{remark}
\newtheorem*{remark*}{Remark}
\newtheorem{remark}[theorem]{Remark}

\newcommand{\R}{\mathbb{R}}
\newcommand{\N}{\mathbb{N}}

\newcommand{\D}{\mathcal{D}}

\newcommand{\st}{\;\ifnum\currentgrouptype=16 \middle\fi|\;}

\usepackage{tikz}
\usetikzlibrary{cd}
\newcommand{\x}{\overline{x}}

\title{Discretization of Dirac systems and port-Hamiltonian systems: the role of the constraint algorithm.}

 \author[1]{Mar\'ia\ Barbero-Li\~n\'an}
 \author[2]{Juan Manuel L\'opez Medel}
\author[2]{David\ Mart\'{\i}n de Diego}

\affil[1]{\small Departamento de Matem\'atica Aplicada, Universidad Polit\'ecnica de Madrid, Av. Juan de Herrera 4, 28040 Madrid, Spain. }
\affil[2]{\small Instituto de Ciencias Matem\'aticas (CSIC-UAM-UC3M-UCM), C/Nicol\'as Cabrera 13-15, 28049 Madrid, Spain.}

\date{\today}

\begin{document}

\maketitle

\abstract{We study the discretization of (almost-)Dirac structures using the notion of retraction and discretization maps on manifolds. Additionally, we apply the proposed discretization techniques to obtain numerical integrators for port-Hamiltonian systems and we discuss how to merge the discretization procedure and the constraint algorithm associated to systems of implicit differential equations.  }

   \vspace{2mm}


    \vspace{2mm}
    

\tableofcontents

\section{Introduction}\label{Intro}

One of the most interesting state-of-the-art research area in mathematics is related with the construction of numerical methods preserving geometric properties as for instance, symplectic integrators for Hamiltonian mechanics, methods preserving first integrals or Poisson structures, numerical methods on manifolds... (see \cite{hairer}).
Most of the relevant dynamical systems in classical mechanics are inherently modeled using the above-mentioned geometric structures. Besides symplectic or Poisson structures, it is also interesting to study other geometric structures such as presymplectic ones. Note that any submanifold of a symplectic manifold inherits a presymplectic structure, idea that it is used to model singular Lagrangian systems and the Dirac theory of constraints. 

 All this plethora of geometric structures (symplectic, Poisson, presymplectic) is unified in a geometric object called 
Dirac structure, which was introduced by Courant and Weinstein~\citep{cournat-weinstein} (see also~\citep{courant,Dorfman} for more details). Apart from this unifying point of view, general Dirac structures have proven to be extremely useful in the modeling of several physical systems, in particular, in the definition of port-Hamiltonian systems (meaning a Hamiltonian systems with ``ports'') which describes general forced Hamiltonian systems that can be interconnected through their ports to build complex physical systems~\citep{vander-schum}.

The main objective of this paper is to discretize Dirac structures in order to construct numerical integrators for the dynamics, once different systems are considered. To achieve that objective we use recent  results about retraction and discretization maps and their lifts to tangent and cotangent bundles \cite{Barbero-DMdD}. The discretization of Dirac structures was previously studied in \cite{leok-ohsawa,leok-ohsawa2, caruso2022discretemechanicalsystemsdirac, Abella-Leok} and, recently in \cite{peng2024discretediracstructuresdiscrete}, but our approach presents a new perspective to discrete Dirac structures. In particular, the application to general configuration manifolds using appropriate retraction or discretization maps (as in  \cite{Barbero-DMdD}) is easier using our techniques. Moreover, the preservation of properties like symplecticity is a direct consequence of the notion of cotangent lift of a discretization map.    

In the following first three sections all the previous notions and tools necessary for the paper are introduced: Dirac structures, constraint algorithm and retraction maps. After the preamble, the sections contain the new results of the paper:
\begin{itemize}
 \item A discretization of Dirac structures and systems depending on a prescribed discretization map is provided in Section~\ref{Sec:DiscreteDirac}. That process is valid on general configuration manifolds.
 \item The role of the constraint algorithm  associated to an implicit system is elucidated in Section~\ref{Subsec:2methods}. To apply first the continuous constraint algorithm and then discretize is different from first discretizing and then apply the discrete constraint algorithm. This is clearly shown in the examples in Sections~\ref{Example:PointVortices}
 and~\ref{example:noholonomic}. Numerical experiments are provided in Section~\ref{Example:PointVortices} that compare the efficiency of both methods with a Runge-Kutta method.
 \item As an indirect consequence, we prove in Proposition~\ref{prop:midpointVortices} that the mid-point discretization of the equations of motion of point vortices in two dimensions preserves the symplecticity. 
 \item Two possible strategies for discretizing port-Hamil\-to\-nian systems using discretization maps are provided in Section~\ref{Sec:PortHamil}.
 \item As a final example of the role of the constraint algorithm in discretization methods we discuss the interesting case of nonholonomic dynamics in Section~\ref{example:noholonomic}. That allows us to obtain a geometric integrator preserving exactly the nonholonomic constraints.  
 \end{itemize}
Some future research lines are presented in Section~\ref{Sec:Future}.

\section{Dirac structures} \label{Sec:Dirac}

We first introduce the main notions related to Dirac structures and Dirac systems. More details can be found in~\citep{cournat-weinstein,courant,Dorfman}. 

\subsection{Linear Dirac structures}

Let $V$ be a $n$-dimensional vector space and we denote by $V^*$  its dual space. Define the  non-degenerate symmetric pairing $\ll \cdot, \cdot \gg $ on $V\oplus V^*$  by
\[
\ll (v_1,\alpha_1),(v_2,\alpha_2) \gg=\langle \alpha_1, v_2\rangle + \langle \alpha_2,v_1\rangle\,,
\label{Eq:Df_Sym_Pair}
\]
for $ (v_1,\alpha_1),(v_2,\alpha_2)\in V\oplus V^*$, where $\langle \cdot, \cdot \rangle$ is the natural pairing between a vector space and its dual.
Given a subspace $U$ of $V\oplus V^*$ define the orthogonal subspace relative to the pairing $\ll\; ,\; \gg$ as
\[
U^{\perp}=\{ (v, \alpha)\in V\oplus V^*\; |\; \forall \; (u, \beta)\in U, \ll (v,\alpha),(u,\beta) \gg=0\}\, .
\]

\begin{definition}
A \emph{linear  Dirac structure on $V$} is a subspace $D\subset V\oplus V^*$ such that $D=D^\perp$.
\end{definition}
Moreover,  besides the notion of Dirac structure, it is interesting  to define other  linear subspaces  on $V\oplus V^*$. In particular, a subspace $U\subset V\oplus V^*$ is called:
\begin{enumerate}
\item  \textit{isotropic} if $U \subseteq U^\perp$.
\item  \textit{coisotropic} if $ U^\perp \subseteq U$.
\end{enumerate}
Thus, a vector subspace $D\subset V\oplus V^*$ is a \textit{Dirac structure} on $V$ if and only if it is maximally isotropic,  that is,  ${\rm dim} \, D=n$ and $\ll (v_1,\alpha_1), (v_2,\alpha_2) \gg =0$ for all $(v_1,\alpha_1)$, $ (v_2,\alpha_2)$ in $D$. 

\begin{example}\label{Example:Dirac}
We now describe some interesting examples of Dirac structures:

\begin{enumerate}
\item Let $F$ be a subspace of $V$, the annihilator $F^\circ$ of $F$ is the subspace of $V^*$ defined as follows
\begin{equation*}
F^\circ=\{\alpha \in V^* \st \langle \alpha, v \rangle=0 \; \mbox{ for all }\; v\in F\}.
\end{equation*}
It can be easily proved that $D_F=F\oplus F^\circ$ is a Dirac structure on $V$.
\item On a presymplectic vector space $(V,\omega)$, the graph of the musical isomorphism $\omega^\flat\colon V\rightarrow V^*$ defines a Dirac structure that we denote $D_\omega$:
\begin{equation*}
D_\omega=\{ (v,\alpha)\in V \oplus V^* \st \alpha=\omega^\flat(v)\},
\end{equation*}
where $\omega^\flat(u)(v)=\omega(u,v)$ for all $u$, $v$ in $V$.
\item Let $\Lambda:V^*\times V^*\to \R$ be a bivector on $V$. Then $\sharp_\Lambda:V^*\to V$ is defined as $\langle\beta,\sharp_\Lambda(\alpha)\rangle=\Lambda(\beta,\alpha)$, with $\alpha,\beta\in V^*$, and its graph defines the Dirac structure
\begin{equation*}
D_\Lambda=\{(v,\alpha)\in V \oplus V^* \st v=\sharp_\Lambda(\alpha)\}.
\end{equation*}
\end{enumerate}
\end{example}

The following fundamental result can be found in~\citep{courant}:
\begin{proposition}
Let $D$ be a Dirac structure on $V$. Define the subspace $F_D\subset V$ to be the projection of $D$ on $V$. Let $\omega_D$ be the 2-form on $F_D$ given by $\omega_D (u,v)=\alpha(v)$, where $u\oplus \alpha\in D$. Then $\omega_D$ is a skew-symmetric form on $F_D$. Conversely, given a vector space $V$, a subspace $F\subset V$ and a skew-symmetric  form $\omega$ on $F$,
\begin{equation*}
D_{F,\omega}=\{u\oplus \alpha \st u\in F, \; \alpha(v)=\omega(u,v) \; \mbox{for all} \; v\in F\}
\end{equation*}
is the only Dirac structure $D$ on $V$ such that $F_D=F$ and $\omega_D=\omega$. \label{Prop:SkewFormDirac}
\end{proposition}
\noindent In other words, a Dirac structure $D$ on $V$ is uniquely determined by a subspace $F_D\subset V$ and a 2-form $\omega_D$. The case $F=V$ is the example 2 above. %The set of Dirac structures on $V$ will be denoted by $\Dir(V)$.

\subsection{Dirac structures on a manifold}\label{sec:DiracOnManifolds}

A \emph{Dirac structure} $D$ on a manifold $M$, is a vector subbundle of the Whitney sum $TM \oplus T^*M$ such that $D_x\subset T_xM \oplus T_x^*M$ is a linear Dirac structure on the vector space $T_xM$ at each point $x\in M$. A \textit{Dirac manifold} is a manifold $M$ with a Dirac structure $D$ on $M$.

From Proposition~\ref{Prop:SkewFormDirac}, a Dirac structure on $M$ yields a distribution $F_{D_x}\subset T_xM$, whose dimension is not necessarily constant, carrying a 2-form $\omega_D(x)\colon F_{D_x}\times F_{D_x} \rightarrow \mathbb{R}$ for all $x\in M$. 
\begin{theorem}\label{Thm:DiracStructM}
Let $M$ be a manifold,  $\omega$ be a 2-form on $M$ and $F$ be a regular distribution on $M$. Define the skew-symmetric bilinear form $\omega_F$ on $F$ by restricting $\omega$ to $F\times F$. For each $x\in M$, define
\begin{align*}
D_{\omega_F }(x)=\left\{(v_x,\alpha_x)\in T_xM \oplus T_x^*M \right. \st& v_x\in F_x, \; \alpha_x(u_x)=\omega_F(x)(v_x,u_x) \; \\
& \left. \mbox{for all } \; u_x\in F_x\right\}.
\end{align*}
Then $D_{\omega_F}\subset TM \oplus T^*M$ is a Dirac structure on $M$. In fact, it is the only Dirac structure $D$ on $M$ satisfying $F_x=F_{D_x}$ and $\omega_F(x)=\omega_D(x)$ for all $x\in M$.
\end{theorem}

\noindent As usual, we have used the terminology \emph{regular distribution} to mean that $F$ has constant rank.  Examples of Theorem~\ref{Thm:DiracStructM} are the case $\omega=0$ where $D_{\omega_F}=F\oplus F^\circ\subset TM\oplus T^*M$, and the case $F=TM$ where $D_\omega$ is the graph of $\omega$.

The dual version of Theorem~\ref{Thm:DiracStructM} is as follows.

\begin{theorem}\label{Thm:DiracStructMDual} Let $M$ be a manifold and let $\Lambda\colon T^*M\times T^*M\rightarrow \mathbb{R}$ be a skew-symmetric two-tensor. Given a regular codistribution $F^{(*)}\subseteq T^*M$ on $M$, define the skew-symmetric two-tensor $\Lambda_{F^{(*)}}$ on $F^{(*)}$ by restricting $\Lambda$ to $F^{(*)}\times F^{(*)}$. For each $x\in M$, let
\begin{align*}
D_{\Lambda_{F^{(*)}}}(x)=\left\{(v_x,\alpha_x)\in T_xM\times T_x^*M \right. &\st \alpha_x\in F^{(*)}_x,\; \beta_x(v_x) = \Lambda_{F^{(*)}}(x)(\beta_x,\alpha_x) \; \\
& \left. \mbox{for all } \; \beta_x\in F^{(*)}_x\right\},
\end{align*}
then $D_{\Lambda_{F^{(*)}}}\subset TM \oplus T^*M$ is a Dirac structure on $M$.
\end{theorem}

\noindent As an example, let $(M,\Lambda)$ be a Poisson manifold.  If $F^{(*)}=T^*M$, then the Dirac structure defined in Theorem~\ref{Thm:DiracStructMDual} is the graph of the Poisson structure considered as a map from $T^*M$ to $TM$.

\begin{remark} A Dirac structure $D$ on $M$ is called \textit{integrable} (see~\citep{courant}) if the condition
\[
\langle {\rm L}_{X_1} \alpha_2, X_3 \rangle + \langle {\rm L}_{X_2} \alpha_3, X_1 \rangle +\langle {\rm L}_{X_3} \alpha_1, X_2 \rangle=0
\]
is satisfied for all pairs of vector fields and 1-forms $(X_1,\alpha_1)$, $(X_2,\alpha_2)$, $(X_3,\alpha_3)$ in $D$, where ${\rm L}_X$ denotes the Lie derivative along the vector field $X$ on $M$. This condition is linked to the notion of closedness for presymplectic forms and Jacobi identity for brackets, and it is sometimes included in the definition of Dirac structure. The integrability condition is too restrictive to describe, for instance, nonholonomic systems, and, for this reason, we do not include the integrability in the general definition of a Dirac structure and in the notion of discretization.

\end{remark}

\subsection{Dirac systems}\label{section:Diracsystems}

As mentioned in Section~\ref{Intro}, Dirac structures $D\subset TM\oplus T^*M$ are very interesting to describe mechanical systems. In concrete, if we additionally give a Hamiltonian function
$H: M\rightarrow {\mathbb R}$ we can write an implicit Hamiltonian system of the form 
\begin{equation}\label{def: Diracsystem}
\dot{x}\oplus dH(x)\in D_x\, .
\end{equation}
The pair $(D, H)$ determines a \textit{Dirac system}. Such a system defines an implicit system of differential equations determined by the submanifold
\[
S=\{v_x\in T_xM\; |\; (v_x, dH(x))\in D_x\}\, .
\]
 Dirac systems are not always defined by Hamiltonian functions. They can also be determined, for instance, by a Lagrangian submanifold ${\mathcal L}$ of $(T^*M, \omega_M)$ using Morse families (generating functions). Such systems are useful for Lagrangian mechanics, optimal control problems, etc... (see~\cite{cendra1,iglesias2}).
We will refer to this case as a (generalized) Dirac system $(D, {\mathcal L})$.

\section{The constraint algorithm}\label{section: constraint}

In general an implicit differential equation on a manifold $M$ can be described as a submanifold $S\subset  TM$. The problem of integrability consists in  identifying  a subset $S_f\subseteq S$ where for any $v\in  S_f$  there exists at least a curve $\gamma: I\subseteq \R\rightarrow M$ such that $\gamma'(0)=v$ and 
$\gamma'(t)\in S_f$ for all $t\in I$. The  algorithm for extracting the integrable part of an implicit differential equation is called a constraint algorithm \cite{MemaTu1}.

    Let $M$ be a manifold and let $S$ be a submanifold of $TM$ describing a system of implicit differential equations. Denote the initial submanifold by 
    $S_0=S$. First,  we project it onto $M$ using the canonical tangent projection $\tau_M: TM\rightarrow M$, that is,
    \[
    M_0=\tau_M(S_0)\,.
    \] The next step is to consider the subset of $TM$ described by the intersection of the initial submanifold and the tangent bundle of $M_0$ to ensure that the solution evolves tangently to the manifold where it lives. In other words,  $S_1=S_0\cap TM_0$. The process is iterated and a sequence of subsets is obtained:
    \begin{align*}
    & S_0	\supseteq S_1 \supseteq  \dots  \supseteq S_{k-1}  \supseteq  S_{k}\, ,\\
    & M_0	\supseteq M_1 \supseteq  \dots  \supseteq M_{k-1}  \supseteq  M_{k}\, ,
    \end{align*}
    where $S_k=S_{k-1}\cap TM_{k-1}$ and 
    $M_{k}=\tau_M(S_k)$ for all $k$. 
    The algorithm stabilizes when there exists $\bar{k}\in \N$ such that  $S_{\bar{k}}=S_{\bar{k}-1}=:S_f$. Then, $S_f$ is the integrable part of $S$ that could possibly be an empty set.

\section{Retraction maps}

The notion of retraction map can be reviewed with more details in~\cite{AbMaSeBookRetraction}. 
Let $M$ be a smooth manifold and $TM$ its tangent bundle. The tangent space at any point $x \in M$ will be denoted by $T_xM$.
\begin{definition}
     A smooth mapping $R: TM \to M$ is called a \textbf{retraction map} on $M$ if it satisfies the following properties:
\begin{itemize}
    \item \textbf{(R1)} $R(0_x)=x$ and
    \item \textbf{(R2)} $T_{0_x}R_x=id_{T_xM}$ for every $x \in M$,
\end{itemize}
where $R_x:=\left.R\right|_{T_xM}$, $0_x$ denotes the zero vector in $T_xM$ and $id_{T_xM}$ denotes the identity map on $T_xM$.
\end{definition}
\noindent Here we have canonically identified $T_{0_x}(T_xM)$ with $T_xM$. Let us take a look at a few examples of retraction maps. 
%Note that $v\in TM$ such that $v_x \in T_xM$ for every $x \in M$.
\begin{remark}\label{eg1}
    On $\mathbb{R}^n$, we define a retraction map simply as a point on the line passing through $x \in \mathbb{R}^n$ in the direction $v_x \in T_x\mathbb{R}^n \cong \mathbb{R}^n$ as $R(v_x)=x+v_x \in \mathbb{R}^n$.
\end{remark}
\begin{remark}\label{eg2}
    On a Riemannian manifold $(M,g)$, we define a retraction map using the exponential map as $R(v_x)=\exp_x(v_x)$ where $\exp_x(v_x)$ is a point on the geodesic passing through $x$ with velocity $v_x$.
\end{remark}
\begin{remark}\label{eg3}
     On a Lie group $G$, we can define a retraction map using the exponential map, see \cite{AbMaSeBookRetraction}, as $R(v_g)=L_g(\exp(T_gL_{g^{-1}}(v_g)))$ where $L_g: G \to G$ denotes the left translation by $g \in G$.
\end{remark}

\subsection{Discretization maps}
A discretization map is a further generalization of a retraction map, see \cite{Barbero-DMdD}. Unlike a retraction map, a discretization map takes $TM$ to two copies of $M$ and hence can be used to develop numerical integrators on $M$ as we shall see in the sequel.
\begin{definition}\label{def:discret}
    A smooth mapping $R_d: TM \to M \times M$ is called a \textbf{discretization map} on $M$ if it satisfies the following properties:
\begin{itemize}
    \item \textbf{(D1)} $R_d(0_x)=(x,x)$ and
    \item \textbf{(D2)} $T_{0_x}R_x^2-T_{0_x}R_x^1=id_{T_xM}$ for every $x \in M$,
\end{itemize}
where $R_d(v_x):=(R^1(v_x),R^2(v_x))$ for every $v_x \in T_xM$ and $R_x^i:=\left.R^i\right|_{T_xM}$ for $i=1,2$.
\end{definition}

\begin{proposition}\label{prop1}\cite{Barbero-DMdD}
    A discretization map $R_d: TM \to M \times M$ is locally invertible around the zero section of $TM$.
\end{proposition}

\begin{remark}
For symplicity, we will assume that the discretization maps are global diffeomorphims between $TM$ and $M\times M$. In general, we would need to work with a tubular section $U$  of the identity section.  
\end{remark}

\begin{example}
    On $\mathbb{R}^n$, we define a discretization map as $R_d(v_x):=(x-\theta v_x,x+(1-\theta)v_x)$ for every $\theta \in [0,1]$. For $\theta=0$, we get the explicit Euler method while for $\theta=0.5$, we get the implicit midpoint rule.
\end{example}
\begin{example}
     On a Riemannian manifold $(M,g)$, we define a discretization map as $R_d(v_x):=(exp(-\theta v_x),exp((1-\theta)v_x))$ for every $\theta \in [0,1]$.
\end{example}
\begin{example}
     On a Lie group $G$, we define a discretization map as $$R_d(v_g):=(L_g(exp(-\theta T_gL_{g^{-1}}(v_g))),L_g(exp((1-\theta) T_gL_{g^{-1}}(v_g))))$$ for every $\theta \in [0,1]$. Here $L_g$ denotes the left translation.
\end{example}

\subsection{Cotangent lift of discretization maps}\label{subsection:cotanget}

We want to define a discretization map on $T^*Q$, that is, 
$R^{T^*}_d: TT^*Q \rightarrow T^*Q\times T^*Q$. The domain lives where the Hamiltonian vector field takes value. Such a map will be obtained by cotangently lifting a discretization map $R_d\colon TQ\rightarrow Q\times Q$ so that the construction $R^{T^*}_d$ will be a symplectomorphism. In order to do that, we need the following three symplectomorphisms (see \cite{Barbero-DMdD} for more details): 
\begin{itemize}
\item The cotangent lift of a diffeomorphism $F: M_1\rightarrow M_2$ defined by:
	\begin{equation*} 
	\hat{F}: T^*M_1 \longrightarrow  T^*M_2 \mbox{ such that } 
\hat{F}=(TF^{-1})^*.
	\end{equation*}
	\item The canonical symplectomorphism:
	\begin{equation*} \alpha_Q\colon T^*TQ  \longrightarrow  TT^*Q  \mbox{ such that }  \alpha_Q(q,v,p_q,p_v)= (q,p_v,v,p_q).
	\end{equation*}

	\item  The symplectomorphism between $(T^*(Q\times Q), \omega_{Q\times Q})$     and   
	$(T^*Q\times T^*Q, \Omega_{12}=pr_2^*\omega_Q-pr^*_1\omega_Q)$:
		\begin{equation*}
	\Phi:T^*Q\times T^*Q \longrightarrow T^*(Q\times Q)\; , \; 
	\Phi(q_0, p_0; q_1, p_1)=(q_0, q_1, -p_0, p_1).	\end{equation*}
	\end{itemize}
The following diagram summarizes the construction process from $R_d$ to $R_d^{T^*}$:
	\begin{equation*}
\xymatrix{ {{TT^*Q }} \ar[rr]^{{{R_d^{T^*}}}}\ar[d]_{\alpha_{Q}} && {{T^*Q\times T^*Q }}  \\ T^*TQ \ar[d]_{\pi_{TQ}}\ar[rr]^{	\widehat{R_d}}&& T^*(Q\times Q)\ar[u]_{\Phi^{-1}}\ar[d]^{\pi_{Q\times Q}}\\ TQ \ar[rr]^{R_d} && Q\times Q }
\end{equation*}

\begin{proposition}\cite{Barbero-DMdD}
	Let $R_d\colon TQ\rightarrow Q\times Q$ be a discretization map on $Q$. Then $${{R_d^{T^*}=\Phi^{-1}\circ \widehat{R_d}\circ \alpha_Q\colon TT^*Q\rightarrow T^*Q\times T^*Q}}$$ 
is a discretization map  on $T^*Q$.
\end{proposition}

\begin{corollary}\cite{Barbero-DMdD} The discretization map
 ${{R_d^{T^*}}}=\Phi^{-1}\circ 	\widehat{R_d}\circ \alpha_Q\colon T(T^*Q)\rightarrow T^*Q\times T^*Q$ is a symplectomorphism between $(T(T^*Q), {\rm d}_T \omega_Q)$ and $(T^*Q\times T^*Q, \Omega_{12})$.
\end{corollary}
\noindent In local coordinates $(q,p,\dot{q},\dot{p})$ for $T(T^*Q)$, the symplectic form ${\rm  d}_T\omega_Q={\rm d}q \wedge {\rm d}\dot{p}+{\rm d}\dot{q}\wedge {\rm d}p$.
\begin{example}\label{example3} On $Q={\mathbb R}^n$ the discretization map 
	$R_d(q,v)=\left(q-\frac{1}{2}v, q+\frac{1}{2}v\right)$ is cotangently lifted to
		$$R_d^{T^*}(q,p,\dot{q},\dot{p})=\left( q-\dfrac{1}{2}\,\dot{q}, p-\dfrac{\dot{p}}{2}; \; q+\dfrac{1}{2}\, \dot{q}, p+\dfrac{\dot{p}}{2}\right)\, .$$
\end{example}

\section{Discretization of Dirac structures}\label{Sec:DiscreteDirac}
Given a discretization map $R_d: TM\rightarrow M\times M$ 
we define the product space 
$$(M\times M)\underset{R_d^{-1}, \pi_M}{\times} T^*M=\{((x_1, x_2), \alpha_x)\; |
\; \tau_M(R_d^{-1}(x_1, x_2))=\pi_M(\alpha_x)\} \, .
$$
In the sequel, we will denote by 
$$(M\times M)^{R_d} \oplus T^*M\equiv (M\times M)\underset{R_d^{-1}, \pi_M}{\times} T^*M\, .
$$

\begin{definition}
Given a Dirac structure $D$ on $M$ we define the {\bf discrete Dirac structure}
$D_d$ as the submanifold of $(M\times M)^{R_d} \oplus T^*M$ given by
\[
D_d=\{((x_1, x_2), \alpha_x)\in (M\times M)^{R_d} \oplus T^*M\; |\; 
(R_d^{-1}(x_1, x_2), \alpha_x)\in D\} \, .
\]
\end{definition}

However, $D_d$ is not a Dirac structure as defined in Section~\ref{sec:DiracOnManifolds}, but it is defined from a Dirac structure. 

\begin{example} The discretization of the third case in Example~\ref{Example:Dirac} is developed.
    On $\mathbb{R}^n$, consider a bivector $\Lambda$  on ${\mathbb R}^n$ and using the midpoint discretization we obtain  $R_d(v_x):=(x-\frac{1}{2} v_x,x+\frac{1}{2}v_x)$: 
\[
D_d=\left\{(x_k, x_{k+1}, \alpha_{x_{k+1/2}})\; |\ 
R_d^{-1}(x_k, x_{k+1})=\Lambda (x_{k+1/2}) \alpha_{x_{k+1/2}}\right\}\, ,
\]
 where $x_{k+1/2}=\frac{x_k+ x_{k+1}}{2}$.\\

%por lo escrito en el ejemplo 2.2, $\Lambda$ se evalúa en pares de covectores. Según la notación descrita allí, sería $R_d^{-1}(x_k, x_{k+1})=\#_\Lambda(\alpha_{x_k+1/2})$, o bien, escribirle el bemol a $\Lambda$
\end{example}
\begin{example}\label{example-sphere}
   Consider the unit sphere $S^{n-1}$ endowed with the  Riemannian metric  $g$ obtained by embedding $S^{n-1}$ in ${\mathbb R}^n$ (with the canonical metric), then the discretization map associated to the Riemannian exponential is given by
   \[
   R_d(x, \xi)=\left(x, x\cos ||\xi||+\frac{\sin ||{x}||}{||\xi||}\xi\right)\, ,
   \]
   with inverse map $R_d^{-1}(x,y)=(x,log^g_x(y))$, where the logarithmic map is given by
   \[
   log^g_x (y)=\arccos{\langle x, y\rangle}\frac{P_x(y-x)}{||P_x(y-x)||}\, ,
   \]
     where
     $P_x (v)=(I-xx^T)v$ and $v$ is a column vector. 
If $\Lambda$ is an almost Poisson tensor on the $S^{n-1}$, then the discrete Dirac structure is given by 
\[
D_d=\left\{(x_k, x_{k+1}, \alpha_{x_{k}})\; |\ 
log^g_{x_{k}}(x_{k+1})=\Lambda (x_{k}) \alpha_{x_{k}}\right\}.
\]
 Another option is to consider as discretization map $\widetilde{R}_d\colon TS^{n-1}\rightarrow S^{n-1} \times S^{n-1}$ given by
 \begin{equation}\label{wideRdSn}
 \widetilde{R}_d(x, v)=\left(\frac{x-v/2}{\|x-v/2\|}, \frac{x+v/2}{\|x+v/2\|}\right)\, ,
 \end{equation}
whose inverse map is:
 \[
\widetilde{R}_d^{-1}(x_k, x_{k+1})=
 \left(\frac{x_k+x_{k+1}}{\|x_k+x_{k+1}\|}, \frac{2(x_{k+1}-x_k)}{\|x_k+x_{k+1}\|}\right)\, .
 \]
Then, the discrete Dirac structure would be
\[
\widetilde{D}_d=\left\{\left.\left(x_k, x_{k+1}, \alpha_{\tfrac{x_k+x_{k+1}}{\|x_k+x_{k+1}\|}}\right)\; \right|\; 
\frac{2(x_{k+1}-x_k)}{\|x_k+x_{k+1}\|}=\Lambda \left(\frac{x_k+x_{k+1}}{\|x_k+x_{k+1}\|}\right) \alpha_{\tfrac{x_k+x_{k+1}}{\|x_k+x_{k+1}\|}} \right\}.
\]
 
\end{example}
\begin{example}
     On a Lie group $G$, we define a discretization map as $$R_d(v_g):=\left(g,L_g(exp( T_gL_{g^{-1}}(v_g)))\right)\equiv (g, g\, exp(g^{-1}v_g)) .$$
 A discrete Dirac structure is given by
\[
D_d=\left\{(g_k, g_{k+1}, \alpha_{g_{k}})\; |\ 
g_k\,exp^{-1}(g_k^{-1}g_{k+1})=\Lambda (g_{k}) \alpha_{g_{k}}\right\},
\]
where $\alpha_{g_k}\in T^*_{g_k}G$. 
\end{example} 
\begin{example}\label{example:canonicaldirac}
Using the cotangent lift $R_d^{T^*}: TT^*Q\rightarrow T^*Q\times T^*Q$ of a discretization map $R_d: TQ\rightarrow Q\times Q$ (see Subsection \ref{subsection:cotanget}) 
we obtain a discrete Dirac structure using the canonical symplectic form $\omega_Q$ in $T^*Q$ as: 
\[
D_d=\left\{(\mu_k, \mu_{k+1}, \alpha_{\mu})\in (T^*Q\times T^*Q)^{R^{T^*}_d} \oplus T^*T^*Q\; |\; 
\ i_{(R^{T^*}_d)^{-1}(\mu_k, \mu_{k+1})}\omega_Q=\alpha_{\mu}\right\},
\]
where $\pi_{T^*Q}(\mu)=\tau_{T^*Q}((R^{T^*}_d)^{-1}(\mu_k, \mu_{k+1}))$.

\noindent For instance, using the cotangent lift in Example~\ref{example3} for $(\mu_k,\mu_{k+1})=(q_k,p_k,q_{k+1},p_{k+1})$ we obtain
the discrete Dirac structure 
\begin{align*}
D_d=\left\{(q_k, p_k, q_{k+1}, p_{k+1}),\right.&  (P_q \,dq+P_p\,dp)_{(q_{k+1/2}, p_{k+1/2})})\; |\; 
q_{k+1/2}=\frac{q_k+q_{k+1}}{2},\\
 & \left. p_{k+1/2}=\frac{p_k+p_{k+1}}{2}, \; q_{k+1}-q_k=P_p; \; p_{k+1}-p_k=-P_q\right\}\, .
\end{align*}
\end{example}
\noindent Similarly, using a discretization map we can define all the corresponding structures (symmetric pairing, isotropic, coisotropic spaces...) on 
$(M\times M)^{R_d} \oplus T^*M$.

\subsection{Discretization of Dirac systems}
A Dirac structure determines a specific relation between cotangent and tangent bundles. This relation is the key to derive the dynamics once a submanifold of the cotangent bundle is provided. Typically, this cotangent bundle is specified given the submanifold $\hbox{Im}\, dH$ where $H: M\rightarrow {\mathbb R}$. Observe that $\hbox{Im}\, dH$ is a Lagrangian submanifold of $(T^*M, \omega_M)$, but other cases are also interesting (specially other types of Lagrangian submanifolds \cite{cendra1,iglesias2}).

Given a submanifold ${\mathcal L}$ of $T^*M$ (typically  ${\mathcal L}$ is a Lagrangian submanifold of $(T^*M, \omega_M)$), a Dirac structure $D$  and a discretization map
$R_d\colon TM \rightarrow M\times M$, we define the \textit{discrete Dirac system} as the subset of $M\times M$ given by
\begin{equation}
S^h_d=\left\{(x_1, x_2)\in M\times M\; |\; 
(\tfrac{1}{h}R^{-1}_d(x_1, x_2), \alpha_x)\in D_x, \alpha_x\in {\mathcal L}, x=\tau_M(R^{-1}_d(x_1, x_2))\right\} \, .
\end{equation}
We introduce the time step $h>0$ since in this paper we are thinking of discretization of continuous system. The product by $1/h$ in $\tfrac{1}{h}R^{-1}_d(x_1, x_2)$ is understood with respect to the vector bundle structure $\tau_{M}: TM\rightarrow M$. That is if $R^{-1}_d(x_1, x_2)=(x, v)$, then $\tfrac{1}{h}R^{-1}_d(x_1, x_2)=(x, \frac{1}{h}v)$. 
\begin{example}
The reduced free rigid body is described by the equations
\begin{equation}\label{eq-q}
\dot{\xi}=\xi\times I^{-1}\xi, \quad \xi\in S^2,
\end{equation}
where 
\[
I=\begin{pmatrix}I_1&0&0\\
0&I_2&0\\
0&0&I_3
\end{pmatrix}
\]
is the inertia tensor. 
In this case the Dirac structure is given by the linear Poisson bivector 
\[
\Lambda_{\xi}=\begin{pmatrix}0&-\xi_3&\xi_2\\
\xi_3
&0&-\xi_1\\
-\xi_2&\xi_1&0
\end{pmatrix}\, .
\]
The Hamiltonian function is 
\[
H(\xi)=\frac{1}{2}\xi \cdot I^{-1}\xi\,.
\]
Therefore, Equation~\eqref{eq-q} is precisely
$
\dot{\xi}=\Lambda_{\xi}dH(\xi)
$.
Using the discretization map $\tilde{R}_d$ in Equation~\eqref{wideRdSn} we obtain the discrete equations (see also \cite{klas}):
\[
\frac{2(\xi_{k+1}-\xi_k)}{h{\|\xi_k+\xi_{k+1}\|}}=\left(\frac{\xi_k+\xi_{k+1}}{\|\xi_k+\xi_{k+1}\|}\right)\times I^{-1}\left(\frac{\xi_k+\xi_{k+1}}{\|\xi_k+\xi_{k+1}\|}\right)\, ,
\]
which could be simplified to 
\[
\frac{\xi_{k+1}-\xi_k}{h}=\left(\frac{\xi_k+\xi_{k+1}}{2}\right)\times I^{-1}\left(\frac{\xi_k+\xi_{k+1}}{\|\xi_k+\xi_{k+1}\|}\right)\, .
\]
\end{example}
\subsubsection{Discretization of a Lagrangian system}\label{example: canonicaldirac system}
Dirac systems can also be defined by a Lagrangian function. The discretization process defined above can also be applied in the Lagrangian framework, even if the Lagrangian function is not regular. For this purpose it is necessary the canonical antisymplectomorphism $ {\mathcal I}_{TQ}$ between the symplectic manifolds $(T^*T^*Q, \omega_{T^*Q})$ and $(T^*TQ, \omega_{TQ})$ ~\citep{XuMac} whose local expression is:
\begin{equation}\label{eq:I_TQlocal}
	   	 {\mathcal I}_{TQ}(q,p, \mu_q, \mu_p)=(q, \mu_p, -\mu_q, p) \, .  
	 \end{equation}
A Lagrangian function  $L: TQ\rightarrow {\mathbb R}$ defines the following Lagrangian submanifold:
\[
{\mathcal L}= {\mathcal I}_{TQ}^{-1}(\hbox{Im} dL)=\left\{(q,p;P_q,P_p)\in T^*T^*Q\; \vert \; p=\frac{\partial L}{\partial \dot{q}}, \; {P}_q=-\frac{\partial L}{\partial q}, \;  P_p= \dot{q}\right\}\, 
\]
of $(T^*T^*Q, \omega_{T^*Q})$. 
If $L$ is a regular Lagrangian, then the Lagrangian submanifold ${\mathcal L}$ is a horizontal Lagrangian submanifold that projects onto the entire $T^*Q$. Locally, the Lagrangian submanifold is defined by  ${\mathcal L}=\hbox{Im} dH$, where $H: T^*Q\rightarrow {\mathbb R}$ is the associated Hamiltonian function \cite{abraham-marsden}. However, when the Lagrangian function $L$ is singular, that is, the Hessian of $L$ with respect to velocities is singular, the Lagrangian submanifold ${\mathcal L}$ is not horizontal. This fact determines the starting point of a constraint algorithm. See Section \ref{section: constraint} for more details.  

In the Lagrangian framework, the discrete Dirac structure ${\mathcal D}$ introduced in Example \ref{example:canonicaldirac} can be used to obtain the following discrete Dirac system 
\[
{S}^h_d=\left\{(\mu_k, \mu_{k+1})\in (T^*Q\times T^*Q)\; |\; 
\ i_{\frac{1}{h}(R^{T^*}_d)^{-1}(\mu_k, \mu_{k+1})}\omega_Q\in {\mathcal L}\right\}\, .
\]
Let $(\mu_k,\mu_{k+1})=(q_k,p_k,q_{k+1},p_{k+1})$, the cotangent lift of the midpoint discretization map leads to the following symplectic integrator
\begin{align*}
  q_{k+1}-q_k&=h\frac{\partial L}{\partial \dot{q}}\left(\frac{q_k+q_{k+1}}{2},\frac{q_{k+1}-q_{k}}{h}\right), \\ 
 p_{k+1}-p_k&=-h\frac{\partial L}{\partial {q}}\left(\frac{q_k+q_{k+1}}{2}, \frac{q_{k+1}-q_{k}}{h}\right)\, .
\end{align*}
Observe that $S^h_d$ defines a Lagrangian submanifold of $T^*Q\times T^*Q$ equipped with the symplectic structure $\Omega_{Q\times Q}=\hbox{pr}_2^*\omega_Q-\hbox{pr}_1^*\omega_Q$ where $\hbox{pr}_a: T^*Q\times T^*Q\rightarrow T^*Q$ are the corresponding projections with $a=1,2$. The Lagrangian character of $S^h_d$ is equivalent to the symplecticity of the implicit map defined $S^h_d$ (see \cite{Barbero-DMdD}, for more details).
\subsection{Two discretization methods for Dirac systems defined by a Hamiltonian a function} \label{Subsec:2methods}%and the constraint algorithm }
After introducing the constraint algorithm in Section~\ref{section: constraint}, let us study how to use the constraint algorithm for Dirac systems in order to obtain numerical integrators for them. As shown in Section \ref{section:Diracsystems}, a Dirac system determined by the pair $(D, {\mathcal L})$, where $D$ is a Dirac structure and ${\mathcal L}$ is a submanifold of $T^*M$, defines
an implicit system as follows
\[
S_0=\{(x,v)\in TM\; |\; (v_x, dH(x))\in D_x\}\,,
\]
if $\mathcal L={\rm Im} {\rm d}H$ for a Hamiltonian function.
To discretize such a Dirac system two different options are considered:  
\begin{enumerate}
 \item {\bf Option 1}: To use a discretization map $R_d: TM\rightarrow M\times M$ in order to obtain a discrete version of $S_0$:   
 $$(S^h_0)_d=\left\{(x_1, x_2)\in M\times M\; |
\; \tfrac{1}{h}R_d^{-1}(x_1, x_2)\in S_0\right\} \, .
$$
\item {\bf Option 2}: First, to apply the constraint algorithm in order to find the integrable part $S_f\subseteq TM_f$ such that $M_f=\tau_M(S_f)$ and $S_f$. Second, to use a discretization map 
$R^f_d: TM_f\rightarrow M_f\times M_f$ to obtain the corresponding discrete structure of $S_f$, that is,   
$$\left(S^d_f\right)_d=\left\{(x_1, x_2)\in M_f\times M_f\; |
\; \tfrac{1}{h}(R^f_d)^{-1}(x_1, x_2)\in S_f\right\} \, .
$$
\end{enumerate}
To illustrate the differences between both approaches, we revisit the case of point vortices in the following section.

\section{A paradigmatic example: Point vortices} \label{Example:PointVortices}
Consider a system of $n$ interacting point vortices in two dimensions
\cite{newton,rowley-marsden}. The equations are given by
\begin{align}\label{eq:point}
\dot{x}^i&=-\frac{1}{2\pi}\sum_{j\not=i}^n \frac{\Gamma_j (y^i-y^j)}{\left(l^{ij}\right)^2},\\
\dot{y}^i&=\frac{1}{2\pi}\sum_{j\not=i}^n \frac{\Gamma_j (x^i-x^j)}{\left(l^{ij}\right)^2},\nonumber
\end{align}
where $l^{ij}=\sqrt{(x^i-x^j)^2+(y^i-y^j)^2}$ are the intervortical distances.
These equations can be expressed in terms of the following singular Lagrangian function
\[
L(x, y, \dot{x}, \dot{y})=\frac{1}{2}\sum_{j=1}^n\Gamma_j (x^j\dot{y}^j-y^j\dot{x}^j)-\frac{1}{4\pi}\sum_{j\not= k}^n \Gamma_j\Gamma_k \log \left(\left(l^{jk}\right)^2\right),
\]
that is,
\begin{equation}\label{eq: lagrangian}
L(q, \dot{q})=\langle \alpha (q), \dot{q}\rangle- H(q),
\end{equation}
where $q=(x,y)\in \R^{2n}$ and  
\begin{align}
\alpha(q)&=\alpha_i (q) {\rm d}q^i=-\frac{1}{2}\Gamma_{ij}y^j dx^i+\frac{1}{2}\Gamma_{ij}x^i dy^j \, ,\label{eq:alpha}\\
H(q)&=\frac{1}{4\pi}\sum_{j\not= k}^n \Gamma_j\Gamma_k \log  \left( \left(l^{jk}\right)^2 \right)\, , \label{eq:h}
\end{align}
where $\Gamma_{ij}=\Gamma_i \delta_{ij}$ are constant.
 The Euler-Lagrange equations for the singular Lagrangian in~\eqref{eq: lagrangian} are
    \begin{equation*}
        \dfrac{d}{dt}(\alpha_i(q))=\dfrac{\partial L}{\partial q^i}(q).
    \end{equation*}
After operating:
    \begin{equation}\label{eq:vortices}
        \dfrac{\partial\alpha_i}{\partial q^j}(q)\, \dot{q}^j=\dfrac{\partial\alpha_j}{\partial q^i}(q) \, \dot{q}^j-\dfrac{\partial H}{\partial q^i}(q),
    \end{equation}
    which are precisely Equations~\eqref{eq:point}. 

Following Subsection~\ref{example: canonicaldirac system}, for any Lagrangian function $L: TQ\rightarrow \R$
we can define
the Lagrangian submanifold
\begin{align*}
{\mathcal L}&= {\mathcal I}_{TQ}^{-1}(\hbox{Im} dL)\\& =\left\{(q,p;P_q,P_p)\in T^*T^*Q \; \mid \; p_i=\alpha_i(q);\;  {P}_{q^i}=-\frac{\partial \alpha_j}{\partial q^i}\dot{q}^j+\frac{\partial H}{\partial q^i}; \; 
 P_{p_i}= \dot{q}^i\right\}\, .
\end{align*}
In the example under consideration $Q=\R^{2n}$. Note that this Lagrangian submanifold does not project onto the entire $T^*\R^{2n}$ because the Lagrangian is singular. Thus, ${\mathcal L}$ is not a horizontal submanifold with respect to the the projection $\pi_{T^*Q}
\colon T^*T^*Q \rightarrow T^*Q$. Let us start the constraint algorithm by taking \begin{equation}\label{eq:S0vortices}
S_0=\sharp^{-1}_{\omega_Q}({\mathcal L})=\left\{(q,p;\dot{q},\dot{p})\in TT^*Q \; \mid \; p_i=\alpha_i(q); \; \dot{p}_i= \frac{\partial \alpha_j}{\partial q^i}\dot{q}^j-\frac{\partial H}{\partial q^i}\right\}\, .
\end{equation}
The steps of the constraint algorithm give us
\begin{align*}
M_0&=\{(q^i, p_i)\in T^*Q \; | \; p_i=\alpha_i(q)\}\subset T^*Q, \\
S_1&=S_0\cap TM_0\\ &=\left\{(q^i,p_i,\dot{q}^i,\dot{p}_i)\; |\; p_i=\alpha_i(q);\;  \dot{p}_i=\dfrac{\partial\alpha_i}{\partial q^j}(q)\dot{q}^j; \; 
\dfrac{\partial\alpha_i}{\partial q^j}(q)\dot{q}^j=\dfrac{\partial\alpha_j}{\partial q^i}(q)\dot{q}^j-\dfrac{\partial H}{\partial q^i}(q)\right\},
   \end{align*} 
   because $$TM_0=\left\{(q^i,p_i,\dot{q}^i,\dot{p}_i)\; |\; p_i=\alpha_i(q); \; \dot{p}_i=\dfrac{\partial\alpha_i}{\partial q^j}(q)\dot{q}^j\right\}.$$
    Note that $\tau_{T^*Q}(S_1)=M_0=M_1$. Hence, the constraint algorithm finishes in the first step.

\[
\begin{tikzcd}
    & TT^*Q \ar[swap]{d} \\
    S_0 \ar[hookrightarrow]{ur} \ar[swap]{dr} & T^*Q  \\
    S_0 \cap TM_0 \ar[hookrightarrow]{u} \ar[hookrightarrow]{d} \ar[rightarrow]{r} & M_0 \ar[hookrightarrow]{u} \ar[rightarrow]{dl} \\
    TM_0  &
\end{tikzcd}
\]

The inclusion $i_{M_i}: M_i\rightarrow T^*Q$ provides every submanifold $M_i$ with a presymplectic 1-form  $$i_{M_i}^*\omega_Q=\omega_{M_i}.$$
In the particular case of point vortices we have that the two-form $\omega_0$ in $M_0\subset \R^{2n}$ is symplectic since
\[
\omega_{M_0}=dq^i\wedge d\alpha_i=-d\alpha\, .
\]
Thus, the dynamics $\Phi: \R^{2n}\rightarrow \R^{2n}$ solution to the differential system 
\[
\dfrac{\partial\alpha_i}{\partial q^j}(q)\dot{q}^j=\dfrac{\partial\alpha_j}{\partial q^i}(q)\dot{q}^j-\dfrac{\partial H}{\partial q^i}(q),
\]
equivalent to Equations~\eqref{eq:point}, 
preserves the symplectic form $\omega_{M_0}$, that is,
\[
\Phi^*({d}\alpha)=d\alpha\, .
\]

\subsection{Method 1: First discretization }

Using the mid-point rule $R_d(q,v)=\left(q-\frac{1}{2}v, q+\frac{1}{2}v\right)$ 
and the corresponding cotangent lift $R_d^{T^*}\colon TT^*Q\rightarrow T^*Q\times T^*Q$, 
		$$R_d^{T^*}(q,p,\dot{q},\dot{p})=\left( q-\dfrac{1}{2}\,\dot{q}, p-\dfrac{\dot{p}}{2}; \; q+\dfrac{1}{2}\, \dot{q}, p+\dfrac{\dot{p}}{2}\right)\, ,$$
whose inverse map is defined by
    \[
    (R_d^{T^*})^{-1}(q_k,p_k,q_{k+1},p_{k+1})=\left(\dfrac{q_k+q_{k+1}}{2},\dfrac{p_k+p_{k+1}}{2},q_{k+1}-q_k,p_{k+1}-p_k\right)\, .
    \]
 For a small step size $h>0$, 

\begin{equation}\label{eq:S0hd}
\left(S^h_0\right)_d=\left\{\ (q_k,p_k,q_{k+1},p_{k+1})\in \mathbb{R}^{2n}\times \mathbb{R}^{2n} \; \left| \; \dfrac{1}{h}\,  (R_d^{T^*})^{-1}(q_k,p_k,q_{k+1},p_{k+1})\in S_0\right.\right\},\end{equation}
 where $S_0$ is defined in Equation~\eqref{eq:S0vortices}. Then, the discrete equations encoded in  
$(S^h_0)_d$ are:
    \begin{align*}
    	\dfrac{p_k+p_{k+1}}{2}&=\alpha\left(\dfrac{q_k+q_{k+1}}{2}\right),\\
    	\dfrac{p_{k+1}-p_k}{h}&=\dfrac{\partial\alpha_j}{\partial q}\left(\dfrac{q_k+q_{k+1}}{2}\right)\left(\dfrac{q_{k+1}^j-q_k^j}{h}\right)-\dfrac{\partial H}{\partial q}\left(\dfrac{q_k+q_{k+1}}{2}\right),
    \end{align*}
    or equivalently
    \begin{align*}
    p_{k}&=\alpha\left(\dfrac{q_k+q_{k+1}}{2}\right)-
\dfrac{h}{2}\, \left(\dfrac{\partial\alpha_j}{\partial q}\left(\dfrac{q_k+q_{k+1}}{2}\right)\left(\dfrac{q_{k+1}^j-q_k^j}{h}\right)-\dfrac{\partial H}{\partial q}\left(\dfrac{q_k+q_{k+1}}{2}\right)\right),\\
p_{k+1}&=\alpha\left(\dfrac{q_k+q_{k+1}}{2}\right)+
\dfrac{h}{2}\, \left(\dfrac{\partial\alpha_j}{\partial q}\left(\dfrac{q_k+q_{k+1}}{2}\right)\left(\dfrac{q_{k+1}^j-q_k^j}{h}\right)-\dfrac{\partial H}{\partial q}\left(\dfrac{q_k+q_{k+1}}{2}\right)\right)   . 
    \end{align*}
From the Lagrangian function in~\eqref{eq: lagrangian}, the following discrete Lagrangian function can be defined (see \cite{rowley-marsden}): 
\[
L_d(q_k, q_{k+1})=h\, L\left(\dfrac{q_k+q_{k+1}}{2}, \dfrac{q_{k+1}-q_k}{h}\right).
\] 
In \cite{rowley-marsden}, the authors prove that this discrete Lagrangian is regular.
%{\textcolor{red}{Do we know for sure that the discrete Lagrangian is regular??? ${\rm det}D_1D_2L_d(q_k,q_{k+1})\neq 0$????}

It can be proved that Equations (\ref{eq:S0hd}) are precisely
\[
p_{k}=-D_1 L_d (q_k, q_{k+1}), \qquad p_{k+1}=D_2 L_d (q_k, q_{k+1})\, .
\]
The well-known discrete Euler-Lagrange equations in~\cite{marsden-west}, $$D_2L_d(q_k,q_{k+1})+D_1L_d(q_{k+1},q_{k+2})=0,$$ become in the example under study: 
\begin{multline*}
    	\frac{1}{2}\dfrac{\partial\alpha_j}{\partial q^i}\left(\dfrac{q_k+q_{k+1}}{2}\right)\left(\dfrac{q_{k+1}^j-q_k^j}{h}\right)+\frac{1}{2}\dfrac{\partial\alpha_j}{\partial q^i}\left(\dfrac{q_{k+1}+q_{k+2}}{2}\right)\left(\dfrac{q^j_{k+2}-q^j_{k+1}}{h}\right)\\
    	-\frac{1}{h}\left(\alpha_i\left(\dfrac{q_{k+1}+q_{k+2}}{2}\right)-\alpha_i\left(\dfrac{q_{k}+q_{k+1}}{2}\right)\right)\\
    	=\frac{1}{2}\dfrac{\partial H}{\partial q^i}\left(\dfrac{q_k+q_{k+1}}{2}\right)+\frac{1}{2}\dfrac{\partial H}{\partial q^i}\left(\dfrac{q_{k+1}+q_{k+2}}{2}\right).
    \end{multline*}
Observe that these equations correspond to a second-order system of difference equations. However, the continuous dynamics in~\eqref{eq:point} is given by a system of first-order differential equations because of the singularity of the continuous Lagrangian function. The use of the cotangent lift of a discretization map to obtain a numerical integrator guarantees that the canonical symplectic form ${\rm d}p_{k+1}\wedge {\rm d}q_{k+1}-{\rm d}p_{k}\wedge {\rm d}q_{k}$ of $T^*Q\times T^*Q$ is preserved. In other words, the discrete flow
\[
\Phi_d\colon T^*Q\longrightarrow T^*Q, \quad \Phi_d(q_k, p_k)=(q_{k+1}, p_{k+1}),
\]
 determined by $\left(S_0^h\right)_d$ in Equation~\eqref{eq:S0hd} is a symplectomorphism. As shown in Section~\ref{Example:PointVortices}, the flow of the continuous system preserves ${\rm d}\alpha$. However, both preservations are only related when $h$ tends to $0$ (see \cite{rowley-marsden} for more details). \\
We define
\[
f(z)=\dfrac{1}{2\pi i}\sum_{j\neq l}\dfrac{\Gamma_l}{z^j-z^l}.
\]
Remembering that $q=(x,y)$ so $z=x+iy$, and writting $z_{k+1/2}=(z_k+z_{k+1})/2$, we have the symplectic method
 \begin{equation}\label{ec:m1pv} \bar{z}^j_{k+2}=\bar{z}^j_k+h\left(f(z_{k+1/2})+f(z_{k+1+1/2})\right).
 \end{equation}
\begin{remark}
As in the continuous case, it is possible to apply a discrete constraint algorithm to the difference equations~\cite{iglesias}. However, both constraint algorithms do not necessarily agree. For instance, in the example of point vortices in Section~\ref{Example:PointVortices}, the continuous constraint algorithm finishes at the first step $S_1$, but the discrete Lagrangian $L_d$ is regular and there is no need to use the discrete constraint algorithm.     
\end{remark}

  \subsection{Method 2: Continuous constraint algorithm plus discretization}  \label{Sec:ConstraintDiscrete}
    Now, we first apply the
continuous constraint algorithm. Then, we discretize using a discretization map. 
From Section~\ref{Example:PointVortices}, we know that  
\begin{align*}
M_f&=\{(q^i, p_i)\in T^*Q \; | \; p_i=\alpha_i(q)\}\subset T^*Q\; , \\
S_f&=S_0\cap TM_0=\left\{(q^i,p_i,\dot{q}^i,\dot{p}_i)\in TT^*Q\; |\; p_i=\alpha_i(q), \; \dot{p}_i=\dfrac{\partial\alpha_i}{\partial q^j}(q)\dot{q}^j,\right.\\
&\qquad \qquad \qquad \qquad \left.\dfrac{\partial\alpha_i}{\partial q^j}(q)\dot{q}^j=\dfrac{\partial\alpha_j}{\partial q^i}(q)\dot{q}^j-\dfrac{\partial H}{\partial q^i}(q)\right\}.
   \end{align*} 
Note that $M_f$ can be identified with the entire manifold $Q={\mathbb R}^{2n}$ because $M_f={\rm Im} \, \alpha$. Analogously, $S_f$ can be projected onto $TQ$ by the tangent map $T\pi_Q\colon TT^*Q \rightarrow TQ$, $T\pi_Q(q,p,\dot{q},\dot{p})=(q,\dot{q})$. Let us denote $T\pi_Q(S_f)$ by $S_f^{TQ}$. Hence, we can directly apply the midpoint rule on $Q$ by means of the discretization map $R_d\colon TQ\rightarrow Q\times Q$, $R_d(q,\dot{q})=(q-\dot{q}/2,q+\dot{q}/2)$, and reconstruct the numerical scheme on $Q$ to obtain the numerical integrator on $T^*Q$.  For a small positive step size $h$, similarly to Equation~\eqref{eq:S0hd}, we have: 
\[
\dfrac{1}{h} \, (R_d)^{-1}(q_k,q_{k+1})\in S_f\subseteq TQ.%\left(S^{TQ}_f\right)_d=R_d\left(S^{TQ}_f\right)\subset Q\times Q\, .
\]
Equivalently, 
\begin{equation}\label{eq:dv2}
    	\dfrac{\partial\alpha_i}{\partial q^j}\left(\dfrac{q_k+q_{k+1}}{2}\right)\left(\dfrac{q_{k+1}^j-q_k^j}{h}\right)
        =\dfrac{\partial\alpha_j}{\partial q^i}\left(\dfrac{q_k+q_{k+1}}{2}\right)\left(\dfrac{q_{k+1}^j-q_k^j}{h}\right)-\dfrac{\partial H}{\partial q^i}\left(\dfrac{q_k+q_{k+1}}{2}\right),
    \end{equation}
    which exactly corresponds to the midpoint discretization of Equations~\eqref{eq:vortices}.\\
    \begin{equation}\label{eq:nmcn}
        \bar{z}^j_{k+1}=\bar{z}^j_k+hf(z_{k+1/2}).
    \end{equation}
In principle, $R_d$ is not designed to preserve any symplectic form such as $d\alpha$. But in this particular case of point vortices dynamics, it can be proved that the midpoint rule preserves the symplectic form $d\alpha$. To prove that statement the following technical result is needed:
\begin{proposition}\label{prop:rdsmp}
    The map
    \[
    R_d:(TQ, d_Td\alpha)\to (Q\times Q,-d\alpha+d\alpha)\footnote{ Here $d_Td\alpha$ denotes the tangent lift of $d\alpha$ and $-d\alpha+d\alpha=-\hbox{pr}_1^*d\alpha +\hbox{pr}_2^*d\alpha$.}
    \]
    with $R_d=(R^1, R^2)$, is a symplectomorphism if the following equations are satisfied:
    \begin{align}\label{eq:rds1}
        \dfrac{\partial^2\alpha_i}{\partial q^j\partial q^k}\dot{q}_k
       & =\left(\dfrac{\partial\alpha_k}{\partial q^l}\circ R^1\right)\dfrac{\partial R^1_k}{\partial q^i}\dfrac{\partial R^1_l}{\partial q^j}-\left(\dfrac{\partial\alpha_k}{\partial q^l}\circ R^2 \right)\dfrac{\partial R^2_k}{\partial q^i}\dfrac{\partial R_l^2}{\partial q^j},\\
   \label{eq:rsd2}
    \dfrac{\partial\alpha_i}{\partial q^j}-\dfrac{\partial\alpha_j}{\partial q^i}
     &=\left(\dfrac{\partial\alpha_k}{\partial q^l}\circ R^1\right)\dfrac{\partial R_k^1}{\partial q^i}\dfrac{\partial R_l^1}{\partial\dot{q}^j}
        -\left(\dfrac{\partial\alpha_k}{\partial q^l}\circ R^2\right)\dfrac{\partial R_k^2}{\partial {q}^i}\dfrac{\partial R_l^2}{\partial \dot{q}^j}
      \\& - \left(\dfrac{\partial\alpha_l}{\partial q^k}\circ R^1\right)\dfrac{\partial R_l^1}{\partial q^j}\dfrac{\partial R_k^1}{\partial\dot{q}^i}
        +\left(\dfrac{\partial\alpha_l}{\partial q^k}\circ R^2\right)\dfrac{\partial R_l^2}{\partial \dot{q}^i}\dfrac{\partial R_k^2}{\partial {q}^j}, \nonumber
   \\ \label{eq:rsd4}
        0&=\left(\dfrac{\partial\alpha_k}{\partial q^l}\circ R^1\right)\dfrac{\partial R_k^1}{\partial \dot{q}^i}\dfrac{\partial R_l^1}{\partial \dot{q}^j}-\left(\dfrac{\partial\alpha_k}{\partial q^l}\circ R^2\right)\dfrac{\partial R_k^2}{\partial \dot{q}^i}\dfrac{\partial R_l^2}{\partial \dot{q}^j}.
    \end{align}
\end{proposition}
\begin{proof}
    The map $R_d$ is symplectic if it is a diffeomorphism and verifies the equation
    \[
    (R_d)^*(-d\alpha+d\alpha)=d_Td\alpha.
    \]
    First, we compute $d_Td\alpha$ knowing that $d\alpha=-\dfrac{\partial\alpha_i}{\partial q^j}dq^i\wedge dq^j$:
    \begin{equation}\label{eq:dTdalpha}
d_Td\alpha=\dfrac{\partial^2\alpha_i}{\partial q^j\partial q^k}\dot{q}_k dq^j\wedge dq^i+\left(\dfrac{\partial\alpha_j}{\partial q^i}-\dfrac{\partial\alpha_i}{\partial q^j}\right)dq^i\wedge d\dot{q}^j.
    \end{equation}

  The pullback $R_d^*\colon \Omega^2(Q\times Q)\rightarrow \Omega^2(TQ)$ of $2$-forms acts as follows:
    \[
    (R_d)^*(-d\alpha+d\alpha)=(R^2)^*(d\alpha)-(R^1)^*(d\alpha).
    \]
For $a=1,2$,
    \[
    (R^a)^*(d\alpha)=-\left(\dfrac{\partial\alpha_i}{\partial q^j}\circ R^a\right)d R_i^a\wedge dR_j^a.
    \]
  Thus,
\begin{equation*}
    \begin{array}{r}
         (R_d)^*(-d\alpha+d\alpha)=\left(\left(\dfrac{\partial\alpha_i}{\partial q^j}\circ R^1\right)\dfrac{\partial R^1_i}{\partial q^k}\dfrac{\partial R^1_j}{\partial q^l}-\left(\dfrac{\partial\alpha_i}{\partial q^j}\circ R^2\right) \dfrac{\partial R^2_i}{\partial q^k}\dfrac{\partial R^2_j}{\partial q^l} \right)dq^k\wedge dq^l\\
        +\left(\left(\dfrac{\partial\alpha_i}{\partial q^j}\circ R^1\right)\dfrac{\partial R^1_i}{\partial q^k}\dfrac{\partial R^1_j}{\partial\dot{q}^l}-\left(\dfrac{\partial\alpha_i}{\partial q^j}\circ R^2\right)\dfrac{\partial R^2_i}{\partial{q}^k}\dfrac{\partial R^2_j}{\partial \dot{q}^l} \right)dq^k\wedge d\dot{q}^l\\
        +\left(\left(\dfrac{\partial\alpha_i}{\partial q^j}\circ R^1\right)\dfrac{\partial R^1_i}{\partial \dot{q}^k}\dfrac{\partial R^1_j}{\partial {q}^l}-\left(\dfrac{\partial\alpha_i}{\partial q^j}\circ R^2\right) \dfrac{\partial R^2_i}{\partial\dot{q}^k}\dfrac{\partial R^2_j}{\partial{q}^l} \right){d\dot{q}^k\wedge d{q}^l}\\
        +\left(\left(\dfrac{\partial\alpha_i}{\partial q^j}\circ R^1\right)\dfrac{\partial R^1_i}{\partial \dot{q}^k}\dfrac{\partial R^1_j}{\partial \dot{q}^l}-\left(\dfrac{\partial\alpha_i}{\partial q^j}\circ R^2\right)\dfrac{\partial R^2_i}{\partial \dot{q}^k}\dfrac{\partial R^2_j}{\partial \dot{q}^l} \right)d\dot{q}^k\wedge d\dot{q}^l.
    \end{array}
\end{equation*}
    We obtain  equations \eqref{eq:rds1} to \eqref{eq:rsd4} matching  the two symplectic forms $d_Td\alpha$ and  $(R_d)^*(-d\alpha+d\alpha)$.
\end{proof}
\begin{proposition}\label{prop:midpointVortices}
The implicit discrete flow $\Phi_d: Q\rightarrow  Q$ induced by Equations~\eqref{eq:dv2} preserves the symplectic form $d\alpha$, that is,
\[
\Phi_d^*({d}\alpha)={ d}\alpha\, ,
\]
if and only if $\alpha$ has  linear components  on $Q$.
\end{proposition}
\begin{proof}
Since
$R_d(q, \dot{q})=(q-\tfrac{1}{2}\dot{q}, q+\tfrac{1}{2}\dot{q})$, the last equation in the proof of Proposition~\ref{prop:rdsmp} becomes
\begin{equation*}
    \begin{array}{r}
         (R_d)^*(-d\alpha+d\alpha)=\left(\left(\dfrac{\partial\alpha_i}{\partial q^j}\circ R^1\right)-\left(\dfrac{\partial\alpha_i}{\partial q^j}\circ R^2\right)  \right)dq^i\wedge dq^j\\
 -\dfrac{1}{2}\left(\left(\dfrac{\partial\alpha_i}{\partial q^j}\circ R^1\right)+\left(\dfrac{\partial\alpha_i}{\partial q^j}\circ R^2\right)\right)dq^i\wedge d\dot{q}^j\\
        -\dfrac{1}{2}\left(\left(\dfrac{\partial\alpha_i}{\partial q^j}\circ R^1\right)+\left(\dfrac{\partial\alpha_i}{\partial q^j}\circ R^2\right)  \right){d\dot{q}^i\wedge d{q}^j}\\     +\dfrac{1}{4}\left(\left(\dfrac{\partial\alpha_i}{\partial q^j}\circ R^1\right)-\left(\dfrac{\partial\alpha_i}{\partial q^j}\circ R^2\right) \right)d\dot{q}^i\wedge d\dot{q}^j.
    \end{array}   
\end{equation*}
Under the assumption of linearity of $\alpha$, that is, $\alpha=\alpha_{ij}q^j{\rm d}q^i$, we have 
\begin{align*}
       (R_d)^*(-d\alpha+d\alpha)&=
 -\alpha_{ij}\,dq^i\wedge d\dot{q}^j
-\alpha_{ij}{d\dot{q}^i\wedge d{q}^j}\\
&=(\alpha_{ji}-\alpha_{ij})dq^i\wedge d\dot{q}^j\\
&=d_Td\alpha\, ,
    \end{align*}   
    because of Equation~\eqref{eq:dTdalpha}. Thus, $R_d$ is a symplectomorphism.
As $S_f$ is a Lagrangian submanifold of $(TQ, d_Td\alpha)$ and $(S_f^{h})_d$ is also a Lagrangian submanifold of $(Q\times Q, -d\alpha+d\alpha)$, the discrete flow on $Q$ preserves the symplectic form $d\alpha$.

\end{proof}

\begin{remark}
    Observe that for discretization maps  of the type $R_d(q,\dot{q})=(q-\theta\dot{q},q+(1-\theta)\dot{q})$, where $\theta\in [0,1]$,  the unique case when $R_d$  is a symplectomorphism is for $\theta=\dfrac{1}{2}$. \\
    %\textcolor{blue}{De hecho, ocurre que para los métodos $R_d^{a,b}(q,\dot{q})=(q-a\dot{q},q+b\dot{q})$, aquellos que $a=b$ son simplectomorfismos.}\\
    We start computing
\begin{equation*}
    \begin{array}{r}
         (R_d)^*(-d\alpha+d\alpha)=\left(\left(\dfrac{\partial\alpha_i}{\partial q^j}\circ R^1\right)-\left(\dfrac{\partial\alpha_i}{\partial q^j}\circ R^2\right)  \right)dq^i\wedge dq^j\\
 \left(-\theta\left(\dfrac{\partial\alpha_i}{\partial q^j}\circ R^1\right)-(1-\theta)\left(\dfrac{\partial\alpha_i}{\partial q^j}\circ R^2\right)\right)dq^i\wedge d\dot{q}^j\\
        \left(-\theta\left(\dfrac{\partial\alpha_i}{\partial q^j}\circ R^1\right)-(1-\theta)\left(\dfrac{\partial\alpha_i}{\partial q^j}\circ R^2\right)  \right){d\dot{q}^i\wedge d{q}^j}\\
        \left(\theta^2\left(\dfrac{\partial\alpha_i}{\partial q^j}\circ R^1\right)-(1-\theta)^2\left(\dfrac{\partial\alpha_i}{\partial q^j}\circ R^2\right) \right)d\dot{q}^i\wedge d\dot{q}^j.
    \end{array}
\end{equation*}
Under the hypothesis of $\alpha$ we need $\theta$ to verify
$
    \theta^2=(1-\theta)^2$ and this implies that $ \theta=\dfrac{1}{2}.
$
\end{remark}

\subsection{Numerical simulations}\label{Sec:numerical}

In this section, we present several numerical experiments in order to gain a deeper understanding of the behavior of the above mentioned Methods 1 and 2.

We first simulate a system of four point vortices, following the initial conditions described in \cite{rowley-marsden} and provided in the following table:
\[
\begin{array}{c|cccc}
    j & 1 &2&3&4 \\
     \hline
     x^j&-1&1&-1&1\\
     y^j&2&2&-2&-2\\
     \Gamma_j&1&1&-1&-1
\end{array}
\]

We fix the timestep to $h=1.0$ and compute 300 steps to visualize the trajectories. We compare the two symplectic methods with the non-symplectic Runge-Kutta 2 integrator. The RK2 method is also used to compute the first step of the other two methods, because they are not self-starting.

\begin{figure}[h]
    \centering
    \includegraphics[scale=0.4]{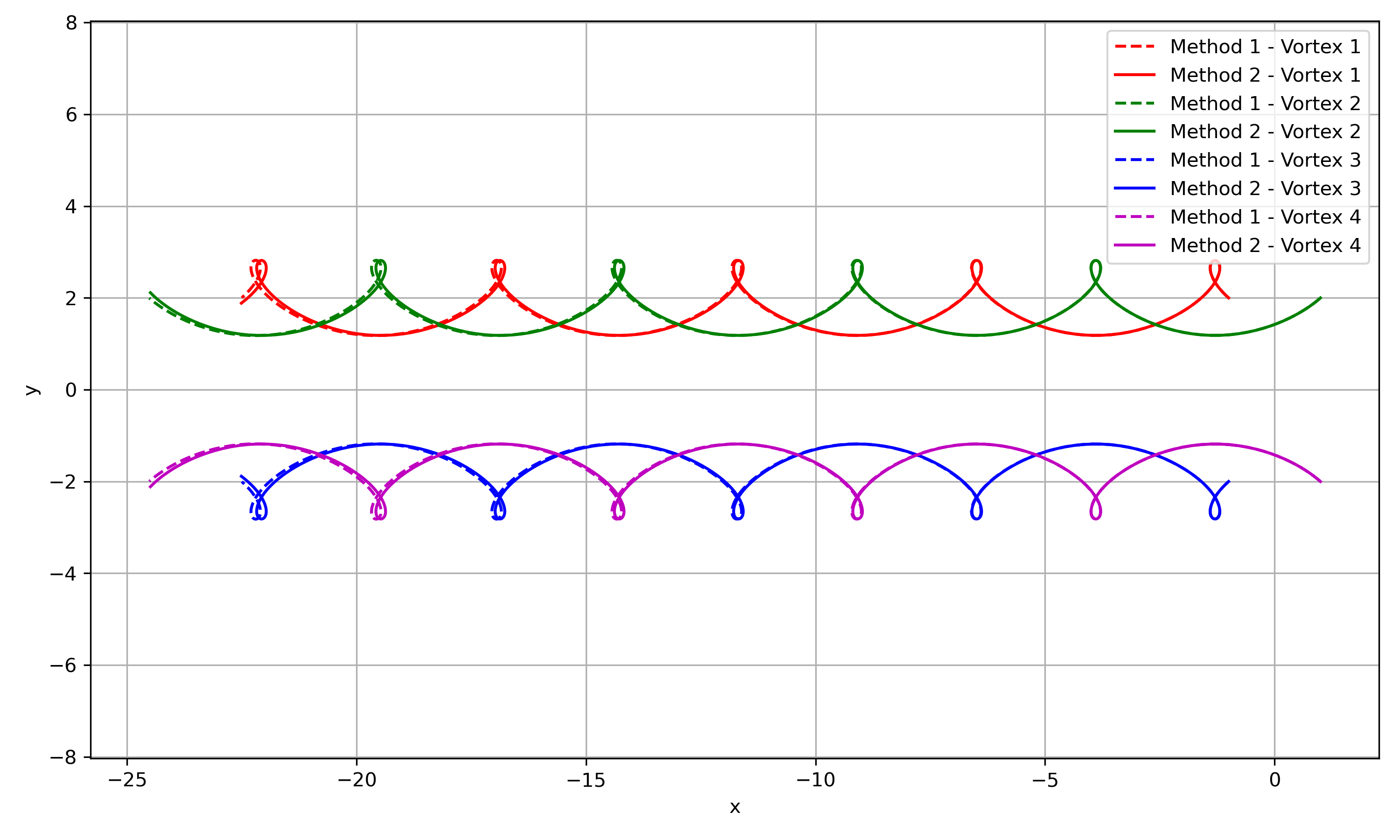}
    \caption{Trajectories of four point vortices obtained with the three numerical methods.}
    \label{fig:frogstep}
\end{figure}

As shown in Figure~\ref{fig:frogstep}, the initial configuration is symmetric with respect to the line $y=0$. The behavior of trajectories is similar under the three numerical methods, with the two pairs of vortices leapfrogging past each other. \\

We analyze energy conservation for both methods by computing the quantity $H(t)-H(0)$ for time $0\leq t\leq 10^6$, see~\eqref{eq:h}. 

\begin{figure}[h]
\centering
\begin{minipage}{0.4\textwidth}
\includegraphics[scale=0.4]{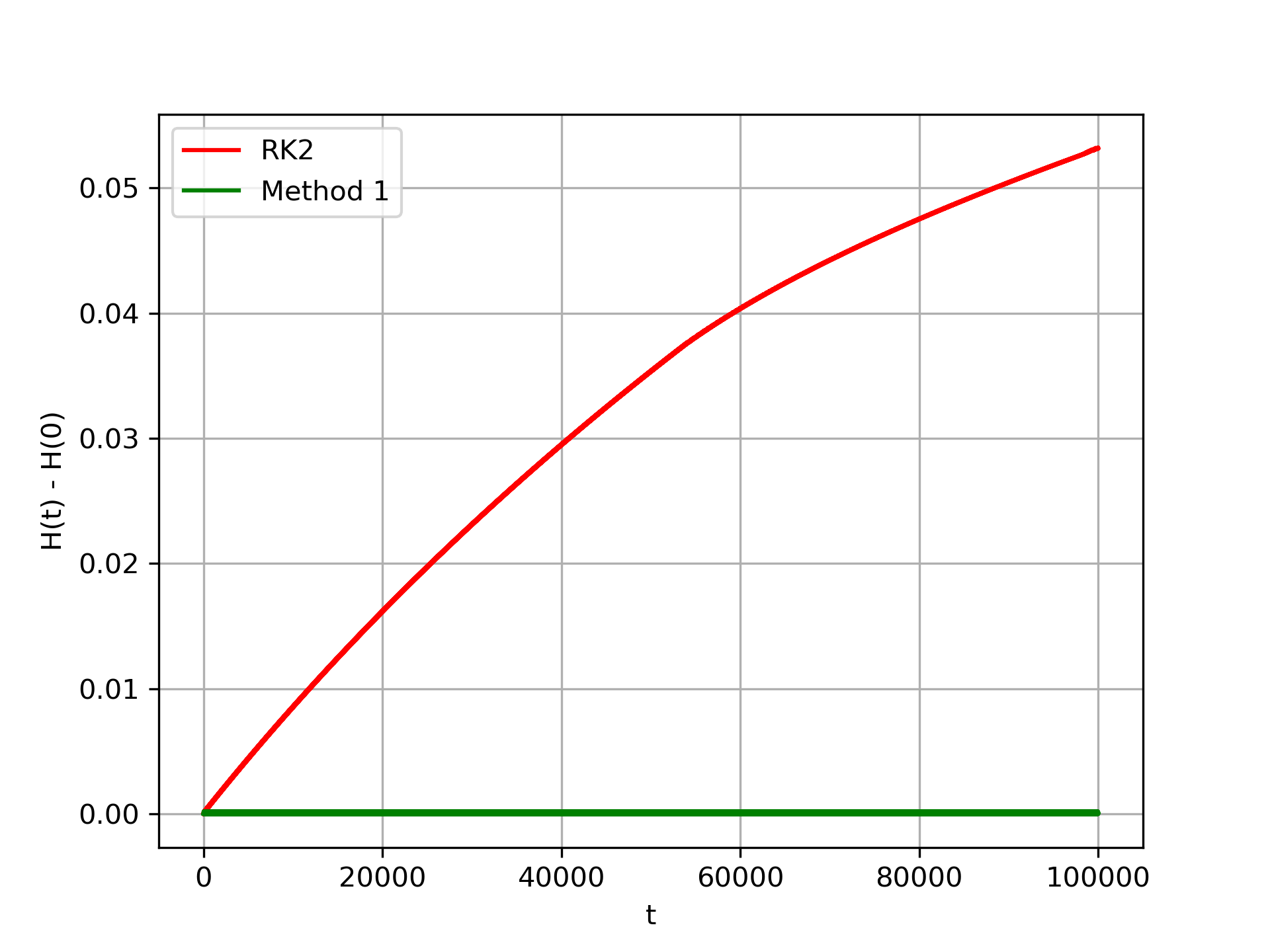}
\end{minipage}\hfill{}\begin{minipage}{0.4\textwidth}
\includegraphics[scale=0.4]{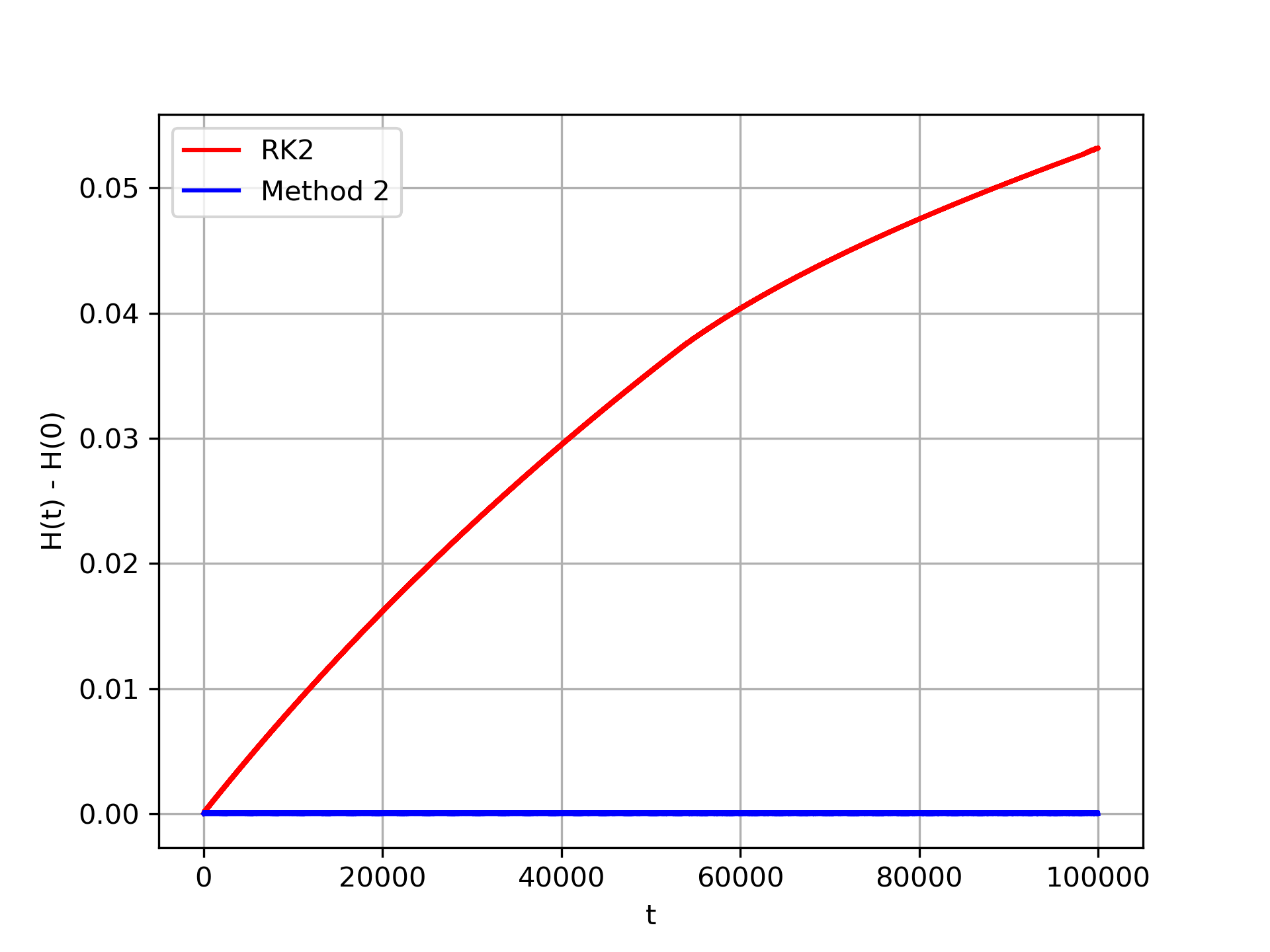}
\end{minipage}
    \caption{Comparison of energy conservation between each of the symplectic method and  RK2.}
    \label{fig:crkm1m2}
\end{figure}

In Figure~\ref{fig:crkm1m2} we can see that RK2 method exhibits a gradual drift, while the symplectic schemes maintain the Hamiltonian close to its initial value at all times.

To compare the two symplectic methods in the paper more clearly, we zoom in on their performance in Figure~\ref{fig:compboth} and increase the number of steps to 500.
The results show that their numerical behavior is similar, although Method 2 shows slightly better accuracy in the performed simulations.

\begin{figure}[h]
    \centering
    \includegraphics[width=0.6\linewidth]{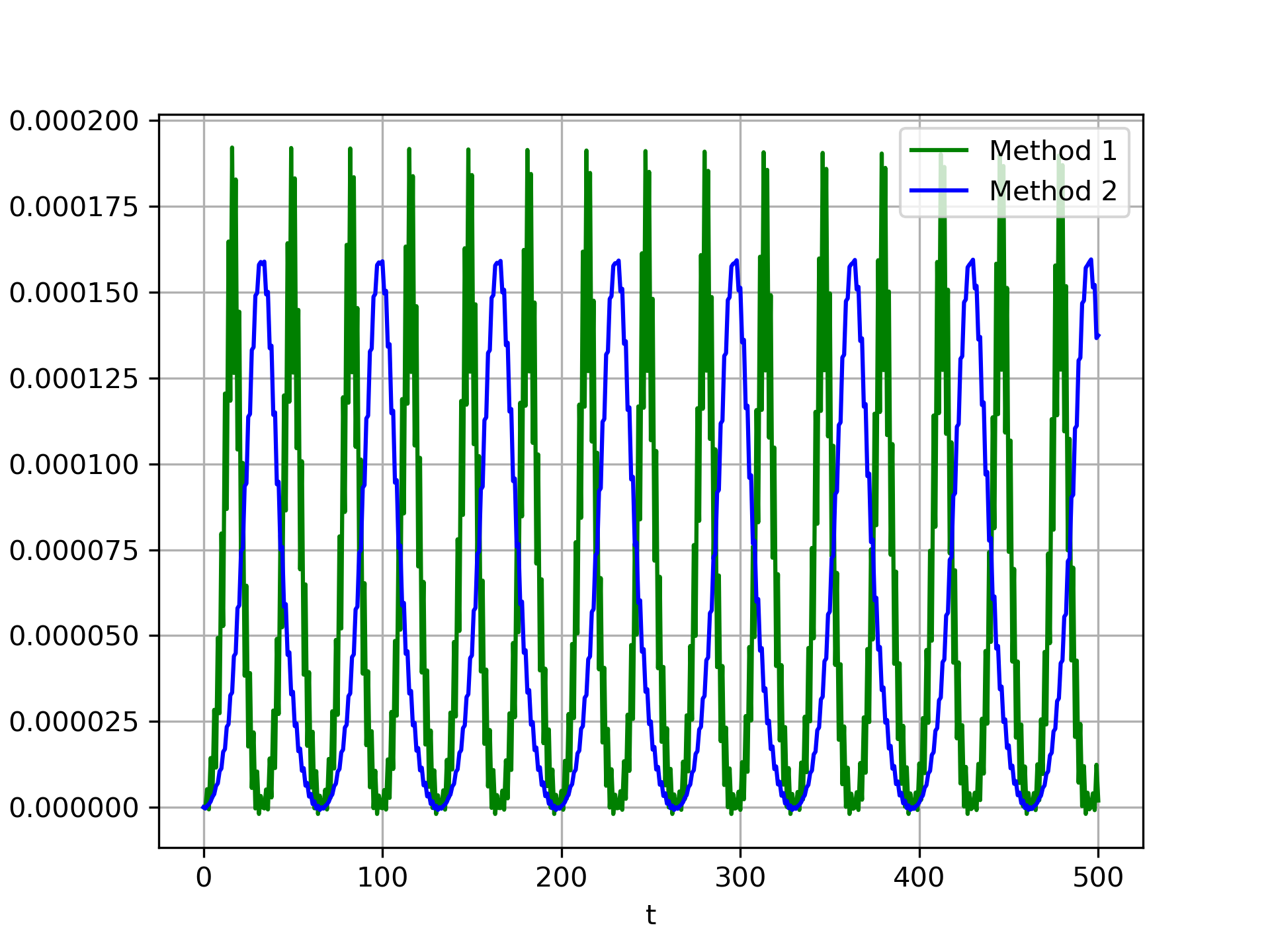}
    \caption{Energy conservation of the symplectic Methods 1 and 2.}
    \label{fig:compboth}
\end{figure}

\section{Application to open and closed port-Hamiltonian systems}\label{Sec:PortHamil}

A  port-Hamiltonian system is specified by a $n$-dimensional configuration space $M$, the spaces of \emph{flows} (inputs), $U=\R^m$, and \emph{efforts} (outputs), $Y=(\R^m)^*=\R^m$, together with the following set of equations in local coordinates $(x,u,y)$ for $M\times U \times Y$:
	\begin{equation}\label{sy:pH}
		\left\{
		\begin{array}{lcl}
			\dot{x}&=&J(x)e+B(x)u,\quad e=dH(x)\, ,\\
			y&=&B^T(x)e\, ,
		\end{array}
		\right.
	\end{equation}
     where $H:M\to\R$, $J(x)$ is a bivector in $\Lambda^2 (T^*M)$  and $B: M\times \R^m\rightarrow TM$ is a vector bundle map over $M$, that is, $\hbox{pr}_1=\tau_{M}\circ B$, with dual vector bundle map $B^T: T^*M\rightarrow M\times \R^m$ over $M$ and we denote its restriction to $x\in M$ by $B^T(x):T_x^*M\to\R^m$. We will use the notation $B(x, u)=B(x) u\in T_xM$ and $B^T(x,\alpha)=B^T(x) \alpha \in (\R^m)^*=\R^m$. 

In geometric terms, the bivector $J$ defines the following Dirac structure $D\subset TM\oplus T^*M$: \begin{equation}
	D:=\{(v,\alpha)\in TM\oplus T^*M\; | \; v=J\alpha\}.\label{set:d}
\end{equation} 
The equations of a port-Hamiltonian system define the following set 
\begin{equation}
	D_{\D,B}:=\{(v,\alpha)\in TM\oplus T^*M\; | \;\forall \; u\in U,\ v-B(x)u=J(x)\alpha\}\, . \label{set:sa}
\end{equation}
An interesting subset of $D_{\D,B}$ is the following one:
\begin{equation}
	D_{\D,B}^{(c)}:=\{(v,\alpha)\in TM\oplus T^*M\; |\;\exists \; u\in U,\ v-B(x)u=J(x)\alpha,\ B^T(x)\alpha=0\}.\label{set:sc}
\end{equation}
Such a port-Hamiltonian system is obtained from closing the ports~\cite{cendra2}.

\begin{proposition} The set $D_{\D,B}$ in Equation~\eqref{set:sa} is coisotropic, while the set $D_{\D,B}^{(c)}$ in Equation~\eqref{set:sc} is a Dirac structure.
\end{proposition}
\begin{proof}
    Consider first the set $D_{\D,B}$. For every element $(v_1,\alpha_1)\in D_{\D,B}$, the elements $(v_2,\alpha_2)$ in the orthogonal complement $(D_{\D,B})_x^\perp$ must satisfy the equality:
    \begin{align*}
        \ll(v_1,\alpha_1),(v_2,\alpha_2)\gg&=\langle\alpha_1,v_2\rangle+\langle\alpha_2,v_1\rangle\\
&=\langle\alpha_1,v_2\rangle+\langle\alpha_2,B(x)u\rangle+\langle\alpha_2,J(x)\alpha_1\rangle=0,\quad \forall \; u\in U.
    \end{align*}
    In particular, for $u=0$ we get the equation
    \[
    \langle\alpha_1,v_2-J(x)\alpha_2\rangle=0 \quad \forall \; (v_1,\alpha_1)\in D_{\D,B} \Leftrightarrow v_2-J(x)\alpha_2=0.
    \]
    Thus,
    \[
    D_{\D,B}^\perp=\{(v_2,\alpha_2)\in TM\oplus T^*M\; |\; \ v_2=J(x)\alpha_2\}\subset D_{\D,B},
    \]
    and is coisotropic.

    Since $D_{\D,B}^{(c)}\subseteq D_{\D,B}$, it is known that $D_{\D,B}^{(c)}$ is also coisotropic. Let us compute the pairing of any two elements  $(v_1,\alpha_1)$,  $(v_2,\alpha_2)$ in $\left(D_{\D,B}^{(c)}\right)_x$: 
    \begin{align*}
    \ll(v_1,\alpha_1),(v_2,\alpha_2)\gg&=\langle\alpha_1,v_2\rangle+\langle\alpha_2,v_1\rangle\\
    &=\langle \alpha_1, J(x)\alpha_2+B(x)u_2\rangle + \langle \alpha_2, J(x)\alpha_1+B(x)u_1\rangle\\
    &=\langle \alpha_1, B(x)u_2\rangle + \langle \alpha_2, B(x)u_1\rangle\\
&=\langle B^T(x)\alpha_1, u_2\rangle + \langle B^T(x)\alpha_2, u_1\rangle=0
    \end{align*}
    Therefore, $D_{\D,B}^{(c)}$ is a Dirac structure. 
\end{proof}

A Hamiltonian function $H\colon M\rightarrow \mathbb{R}$ together with the coisotropic structure $D_{\D, B}$ define the following coisotropic system, known in the literature as an open port-Hamiltonian system,
\begin{equation}\label{sy:opH}
		\left\{
		\begin{array}{lcl}
			\dot{x}&=&J(x) d H(x) +B(x)u\, ,\\
			y&=&B^T(x)dH\, ,
		\end{array}
		\right.
	\end{equation}
or, equivalently,
\begin{equation} \label{eq:IntrinsicOpPH}
\dot{x}\oplus dH(x)\in  (D_{\D, B})_x, \qquad y=B^T(x)(dH(x))\, ,
\end{equation}

On the other hand, a Hamiltonian function $H$ together with the Dirac structure $D_{\D,B}^{(c)}$ define the following Dirac system, also known in the literature as a closed port-Hamiltonian system:
\begin{equation}\label{sy:cpH}
		\left\{
		\begin{array}{lcl}
			\dot{x}&=&J(x) d H(x) +B(x)u\, ,\\
			0&=&B^T(x)dH(x)\, ,
		\end{array}
		\right.
	\end{equation}
or, equivalently,
\begin{equation} \label{eq:IntrinsicClosePH}
\dot{x}\oplus dH(x)\in  (D_{\D, B}^{(c)})_x \; .
\end{equation}

\subsection{Discretization}\label{section:directdiscretization}

A discretization map $R_d:TM\to M\times M$ is used to obtain numerical integrators for the port-Hamiltonian systems mentioned above taking into account the continuous dynamics. Let $\overline{x}=\tau_M(R_d^{-1}(x_k,x_{k+1}))$. The coisotropic or open port-Hamiltonian system in~\eqref{eq:IntrinsicOpPH} is discretized for a small step size $h>0$ as follows (see \cite{KOTYCZKA2019104530}):
\begin{equation} \label{eq:DiscreteIntrinsicOpPH}
\left( \dfrac{1}{h}\, R_d^{-1}(x_k,x_{k+1}) \right) \oplus dH(\overline{x})\in  (D_{\D, B})_{\overline{x}}, \qquad y_{\overline{x}}=B^T(\overline{x})(dH(\overline{x}))\, .
\end{equation}
Equivalently,
\begin{equation}\label{sy:Ded}
    \left\{
    \begin{array}{rcl}
     R_d^{-1}(x_k,x_{k+1})-hB(\x)u_{\x}&=&hJ(\x)dH(\x),\\
y_{\overline{x}}&=&B^T(\x)dH(\x).
    \end{array}
    \right.
\end{equation}
Observe that $u_{\x}$ and $y_{\x}$ represent, respectively,  the discrete flow and discrete efforts associated to this discretization.  

Moreover,
\[
h\langle y_{\bar{x}}, u_{\bar{x}}\rangle =\langle dH (\bar{x}), R_d^{-1}(x_k,x_{k+1})\rangle  \,.
\]

On the other hand, a closed port-Hamiltonian system~\eqref{eq:IntrinsicClosePH} is discretized as follows for a small step size $h>0$:
\begin{equation} \label{eq:DiscreteIntrinsicClosePH}
\left( \dfrac{1}{h}\, R_d^{-1}(x_k,x_{k+1}) \right) \oplus dH(\overline{x}) \in  (D_{\D, B}^{(c)})_{\overline{x}} \; .
\end{equation}
Equivalently,
\begin{equation}\label{sy:Ded-closed}
    \left\{
    \begin{array}{rcl}
     R_d^{-1}(x_k,x_{k+1})-hB(\x)u_{\x}&=&hJ(\x)dH(\x),\\
0&=&B^T(\x)dH(\x).
    \end{array}
    \right.
\end{equation}
For Dirac structures, $\langle y_{\bar{x}}, u_{\bar{x}}\rangle =0$. Thus, 
\[
\langle dH (\bar{x}), R_d^{-1}(x_k,x_{k+1})\rangle=0
\]
which is not equal  to $H(x_{k+1})-H(x_k)$. For guaranteeing exact preservation of the energy along the discrete trajectory it is convenient  to use discrete gradient methods (see \cite{Celledoni-Hoset}).
\begin{remark}
Observe that a closed port-Hamiltonian system can be alternatively rewritten as 
\begin{equation}\label{sy:cpH-1}
			\begin{pmatrix} I\\
        0\end{pmatrix}\dot{x}=
        \begin{pmatrix}J(x) d H(x) +B(x)u\\
			B^T(x)dH(x)
		\end{pmatrix} \, ,
	\end{equation}
where $I$ is the identity matrix. Such a system is  a particular case  of an implicit differential system where it is necessary to apply a constraint algorithm to guarantee the consistency of the solutions of these equations. 
\end{remark}

\subsection{Method 2 for closed port-Hamiltonian systems}

The constraint algorithm can also be applied to a closed port-Hamiltonian system on $M$ as in Equation~\eqref{sy:cpH-1}:
\begin{align*}
\dot{x}&=J(x)dH(x)+B(x)u_{x},\\
0&=B^T(x)dH(x)\,.
\end{align*}
These equations determine the starting submanifold $S_0$ of $TM$ and 
\[
M_0=\tau_M(S_0)=\{x\in M\; |\; B^T(x)dH(x)=0\}\, .
\]
The first step of the algorithm consists of finding the subset $S_1\subseteq TM_0$ given by
\begin{align*}
S_1=S_0\cap TM_0&=\{(x, \dot{x})\; |\; \exists \; u_x\in U \mbox{ such that }
\dot{x}=J(x)dH(x)+B(x)u_{x},\\
&
0=B^T(x)dH(x), \quad \langle d\left(B^T(x)dH(x)\right), \dot{x}\rangle=0\} \, . 
\end{align*}
If we try to discretize the dynamics encoded in $S_1$ as in Section~\ref{section:directdiscretization}, the main difficulty is to find a discretization map on $M_0$. 
It could be constructed by defining a projector $P: M\rightarrow M_0$ from $M$ to $M_0$ such that $P(x)=x$ for any $x\in M_0$ as described in the following diagram:
\begin{center}
    \begin{tikzcd}
        TM \arrow[r,"R_d"] & M \times M \arrow[d,"P\times P"] \\
        TM_0\arrow[u,"Ti_{M_0}"] \arrow[r,"R_d^{M_0}"] & M_0 \times M_0 \, ,
    \end{tikzcd}
\end{center}
where $i_{M_0}: M_0 \hookrightarrow M$ is the inclusion.

\begin{proposition} \label{Prop:RdM0} If $R_d: TM\rightarrow M\times M$ is a discretization map, then 
the mapping $R_d^{M_0}: TM_0 \to M_0 \times M_0$ defined as $R_d^{M_0}:=(P \times P) \circ R_d \circ Ti_{M_0}$ is also a discretization map.
\end{proposition}
\begin{proof}
     First,  we show that  for all $x\in M_0$, $R_d^{M_0}(0_{x})=(x,x)$. That is
\begin{align*}
     R_d^{M_0}(0_{x})&=\left((P \times P) \circ R_d \circ Ti_{M_0}\right)(0_{x})\\
        &=\left((P \times P) \circ R_d \right)(0_{x})\\
        &=(P \times P)(x, x)=
        (x,x)\, ,
\end{align*}
because by definition  $P_{|M_0}={\rm id}_{|M_0}$.

Secondly, it must be proved that $T_{0_{x}}R_d^{M_0,2}-T_{0_{x}}R_d^{M_0,1}=id_{TM_0}$, where $$R_d^{M_0}(X_{x})=(R_d^{M_0,1}(X_{x}),R_d^{M_0,2}(X_{x})).$$ Let us compute: 
\begin{align*}
    &\left(T_{0_{x}}R_d^{M_0,2}-T_{0_{x}}R_d^{M_0,1}\right)(X_{x})\\
    &\hskip-0.6cm=\left.\frac{d}{ds}\right|_{s=0} \Bigl[R_d^{M_0,2}(sX_{x})-R_d^{M_0,1}(sX_{x}) \Bigr]\\
    &\hskip-0.6cm=\left.\frac{d}{ds}\right|_{s=0} \Bigl[P\circ R_d^{2}\circ Ti_{M_0}(sX_{x})-P\circ R_d^{1}\circ Ti_{M_0}(sX_{x})) \Bigr]\\
    &\hskip-0.6cm=TP\biggl[\left.\frac{d}{ds}\right|_{s=0} \Bigl[R_d^{2}\circ Ti_{M_0}(sX_{x})-R_d^{1}\circ Ti_{M_0}(sX_{x}) \Bigr] \biggr]\\
    &\hskip-0.6cm=TP\biggl[\left.\frac{d}{ds}\right|_{s=0} \Bigl[R_d^{2}(sTi_{M_0}(X_{x}))-R_d^{1}(sTi_{M_0}(X_{x})) \Bigr] \biggr]\\
&\hskip-0.6cm=TP\Bigl[T_{0_{x}}R_d^{2}-T_{0_{x}}R_d^{1} \Bigr](Ti_{M_0}(X_{x}))\\
&\hskip-0.6cm=TP\circ Ti_{M_0}(X_{x}) = T(P\circ i_{M_0})X_{x}=X_x
\end{align*}
for all $X_{x} \in T_xM_0$. For the proof we have used that $R_d$ is a discretization map  (see Definition~\ref{def:discret}) and that $P\circ (i_{M_0})_x={\rm id}_x$ because $P$ is a projector.
\end{proof}

Therefore, the discretization of the closed port-Hamiltonian system is given 
by
\begin{align*}
\left(R_d^{M_0}\right)^{-1}(x_k, x_{k+1})&=hJ(x_{k, k+1})dH(x_{k, k+1})+hB(x_{k, k+1})u_{x_{k, k+1}},\\
0&=\Phi_l(x_{k, k+1})=e_l \, B^T(x_{k, k+1})dH(x_{k, k+1}), \\
0&=(R_d^{M_0})^*\left({h}(d_T\phi_l)(x_k, x_{k+1})\right)
\end{align*}
where $x_{k, k+1}=\tau_{M_0}\left(\left(R_d^{M_0}\right)^{-1}(x_k, x_{k+1})\right)$.

\section{A particular case: nonholonomic dynamics}\label{example:noholonomic}. 
We consider a mechanical Lagrangian  $L: TQ\rightarrow {\mathbb R}$ defined by the following data: 
\begin{itemize}
\item A  Riemannian metric $g$  on a $n$-dimensional manifold $Q$ that defines the musical isomorphisms: 
$\flat_g: TQ\rightarrow T^*Q$ is the vector bundle isomorphism defined by 
$\langle \flat_g (v_q), w_q)=g_q(v_q,w_q)$, for all $v_q, w_q\in T_qQ$ and 
the inverse isomorphism denoted by $\sharp_g=(\flat_g)^{-1}: T^*Q\rightarrow TQ$.
The Riemannian metric defines  the kinetic energy ${\mathcal K}_g: TQ \to {\mathbb R}$ on $TQ$ 
by ${\mathcal K}_g(v_q)=\frac{1}{2} g_q(v_q, v_q)$.
\item A potential energy function $V\in C^{\infty}(Q)$.
\end{itemize}
The mechanical Lagrangian function $L: TQ\rightarrow {\mathbb R}$ is given by
\begin{equation}\label{def:lagrangian}
L={\mathcal K}_g  - V\circ \tau_Q.
\end{equation}
Observe that in local coordinates $(q^i, \dot{q}^i)$ for $TQ$, 
\[
L(q,\dot{q})=\frac{1}{2}g_{ij}(q)\dot{q}^i\dot{q}^j-V(q)\, ,
\]
where $g=g_{ij}dq^i\otimes dq^j$.

The classical Euler-Lagrange equations for the Lagrangian $L$ are
\[
\frac{d}{dt}\left(\frac{\partial L}{\partial \dot{q}^i}\right)-\frac{\partial L}{\partial q^i}=0, \quad 1\leq i\leq n\; .
\]

A mechanical nonholonomic system is defined by the triple $(Q, L, {\mathcal D})$ where $L$ is the mechanical Lagrangian  defined in (\ref{def:lagrangian}) 
and ${\mathcal D}$ is a nonintegrable distribution on the configuration
manifold $Q$. 
  The nonintegrable distribution  ${\mathcal D}$ restricts
  the possible   velocity vectors  
without imposing any restriction on the configuration space \cite{Bloch}. 
Locally, the nonholonomic constraints are given by a set of $m\leq n=\dim Q$ equations that are
linear on the velocities
\[
\mu^a_i(q)\dot{q}^i=0, \quad 1\leq a\leq m\; .
\]
The distribution ${\mathcal D}$ defines the vector subbundle ${\mathcal D}^o \subseteq T^*Q$, called the annihilator of ${\mathcal D}$,
spanned at each point by the one forms $\{\mu^a\}$  locally given by $\mu^a=\mu^a_i\, dq^i$. 

The Lagrange-d’Alembert principle  states that the constrained solutions for the mechanical nonholonomic problem $(Q, L, {\mathcal D})$  are those curves on $Q$ satisfying
 the following nonholonomic equations: 
\begin{align*}
\frac{d}{dt}\left(\frac{\partial L}{\partial \dot{q}^i}\right)-\frac{\partial L}{\partial q^i}&=\lambda_a\mu^a_i, \quad 1\leq i\leq n\; ,\\
\mu^a_i(q)\dot{q}^i&=0, \quad 1\leq a\leq m\; , 
\end{align*}
where $\lambda_a$ are Lagrange multipliers determined by taking the time derivative of the  nonholonomic constraints.

The previous equations are equivalent to the following closed port-Hamiltonian equations: 
\begin{align}
\begin{pmatrix}\dot{q}\\ 
\dot{p}\end{pmatrix}&=
\begin{pmatrix}0_{n\times n}&I_{n\times n}\\-I_{n\times n}& 0_{n\times n}\end{pmatrix}
dH+\begin{pmatrix} 0_{n\times m}\,  \\ \mu_{n\times m}\end{pmatrix}\lambda_{m\times 1}\, ,\label{eq.eqsnonh}\\
0&=\begin{pmatrix} 0_{n\times m}\\ \mu_{n\times m}\end{pmatrix}^T dH\, ,\label{eq.constraintsnonh}
\end{align}
where $\mu_{n\times m}$ is the matrix with coefficients $\mu_{ia}=\mu^a_i(q)$ 
and $H: T^*Q\rightarrow {\mathbb R}$ is the corresponding Hamiltonian function
\[
H(q^i, p_i)=\frac{1}{2}g^{ij}(q)p_ip_j+V(q),
\]
where $p_i=g_{ij}(q)\dot{q}^j$. 
Observe that Equations~\eqref{eq.constraintsnonh} are equivalent to
\[
\Phi^a(q, p)=\mu^a_i(q)g^{ij}(q)p_j=\mu^a_i(q)\dot{q}^i=0\, .
\]
\subsection{Method 1: First discretization}\label{ss:m1}
The numerical scheme can be obtained by using a discretization map $R_d: TT^*Q\rightarrow T^*Q\times T^*Q$. We illustrate the method using the cotangent lift of the midpoint rule under the assumption that $Q$ is a vector space. The corresponding discretization is: 
\begin{align}
\frac{q^i_{k+1}-q^i_k}{h}=&g^{ij}(q_{k+1/2})(p_{k+1/2})_j,\\
\frac{(p_{k+1})_i-(p_k)_i}{h}=&-\frac{1}{2}\dfrac{\partial g^{jl}(q_{k+1/2})}{\partial q^i}(p_{k+1/2})_j(p_{k+1/2})_l-\frac{\partial V}{\partial q^i}(q_{k+1/2})\nonumber\\&+\lambda_a\mu^a_i(q_{k+1/2}),\\
0&=\mu^a_i(q_{k+1/2})g^{ij}(q_{k+1/2})(p_{k+1/2})_j,
\end{align}
where $q_{k+1/2}=\tfrac{q_{k+1}+q_k}{2}$ and $p_{k+1/2}=\tfrac{p_{k+1}+p_k}{2}$.
This method is obviously related to the discrete Lagrange-d'Alembert's principle first proposed in~\cite{JorgeSonia}.  

\subsection{Method 2: Continuous constraint algorithm plus discretization}\label{section:method 2}

If we first apply the constraint algorithm, we add the total derivative of the nonholonomic constraints as an additional constraint: 
\[\dfrac{{\rm d}}{{\rm d}t} \Phi^a(q,p)=
\frac{\partial (\mu^a_i g^{ij})}{\partial q^k}g^{kl}p_lp_j+\mu^a_ig^{ij}\dot{p}_j=0\, .
\]
Using the time derivative of the momenta in~\eqref{eq.eqsnonh}, the Lagrange multipliers can be uniquely determined as follows:
\begin{align*}\dfrac{\partial (\mu^a_i g^{ij})}{\partial q^k}g^{kl}p_lp_j&+\mu^a_ig^{ij}\left( -\dfrac{\partial H}{\partial q^j}+\lambda_b\mu^b_j\right)=0\, , \\\dfrac{\partial (\mu^a_i g^{ij})}{\partial q^k}g^{kl}p_lp_j&+\mu^a_ig^{ij}\left( -\dfrac{1}{2}\, \dfrac{\partial g^{rs}}{\partial q^j} p_rp_s-\dfrac{\partial V}{\partial q^j}+\lambda_b\mu^b_j\right)=0\, , \\ \mu^a_ig^{ij}\mu^b_j\lambda_b&=\mu^a_ig^{ij}\, \dfrac{\partial V}{\partial q^j}+\dfrac{1}{2}\mu^a_ig^{ij}\,\dfrac{\partial g^{rs}}{\partial q^j} p_rp_s -
\dfrac{\partial (\mu^a_i g^{ij})}{\partial q^k}g^{kl}p_lp_j\, , \\
\lambda_b(q,p)&=C_{ab}\left(\mu^a_ig^{ij}\, \dfrac{\partial V}{\partial q^j}+\dfrac{1}{2}\mu^a_ig^{ij}\,\dfrac{\partial g^{rs}}{\partial q^j} p_rp_s -
\dfrac{\partial (\mu^a_i g^{ij})}{\partial q^k}g^{kl}p_lp_j\right)\, .
\end{align*}
where $(C_{ab})$ is the inverse matrix of $C^{ab}=\mu^a_i g^{ij}\mu^b_j$.

To define a discretization map on $M_0=\{(q, p)\in T^*Q\; |\; \mu^a_ig^{ij}p_j=0\}$ we use the Riemannian metric $ g$ to define the orthogonal projector ${\mathcal P}\colon T^*Q\rightarrow M_0$:
\[
{\mathcal P}(\alpha)=\alpha-C_{ab}(g^{ij}\alpha_j\mu^a_i)\mu^b\, ,
\] where $\mu^b$ is an element in $T^*Q$. It can be proved that the projector is well-defined using the language of matrices. Proposition~\ref{Prop:RdM0} guarantees the existence of the discretization map $R_d^{M_0}\colon TM_0\rightarrow M_0\times M_0$ defined by:
\begin{align*}
R_d^{M_0}(q,p; \dot{q}, \dot{p})=
\left(
q^-, \right. & p^- - C_{ab}(q^-)g^{ij}(q^-)p^-_j \mu^a_i(q^-)\mu^b(q^-),\\
& \left. q^+,\;  p^+ - C_{ab}(q^+)g^{ij}(q^+)p^+_j \mu^a_i(q^+)\mu^b(q^+) \right) \, ,
\end{align*}
where $q^-=q-\frac{1}{2}\dot{q}$, $p^-=p-\frac{1}{2}\dot{p}$,  $q^+=q+\frac{1}{2}\dot{q}$ and $p^+=p+\frac{1}{2}\dot{p}$. 

As a consequence, we obtain the following implicit method: 

\begin{align}
q_k&=q-\frac{h}{2}g^{ij}(q)p_j\, ,\label{eq-nh-1}\\
p_k&={\mathcal P}_{q_k}\left(p-\frac{h}{2}\left(
-\frac{\partial H}{\partial q}(q,p)+{\lambda}_{\alpha}(q, p)\mu^a_i(q)
\right)\right)\, ,\label{eq-nh-2}\\
0&=\mu^a_i(q)g^{ij}(q)p_j\, ,\label{eq-nh-3}
\\
q_{k+1}&=q+\frac{h}{2}g^{ij}(q)p_j \, ,\label{eq-nh-4}\\
p_{k+1}&={\mathcal P}_{q_{k+1}}\left(p+\frac{h}{2}\left(
-\frac{\partial H}{\partial q}(q,p)+{\lambda}_{\alpha}(q, p)\mu^a_i(q)
\right)\right) \, .\label{eq-nh-5}
\end{align}
The implicit method works as follows.
Given $(q_k, p_k)\in M_0$ that verifies \begin{equation*} \mu^a_i(q_k)g^{ij}(q_k)(p_k)_j=0, \end{equation*} find the unique $(q, p)$ verifying Equations (\ref{eq-nh-1}), (\ref{eq-nh-2}) and (\ref{eq-nh-3}). Then, we obtain the next step $(q_{k+1}, p_{k+1})\in M_0$ substituting in Equations (\ref{eq-nh-4}) and (\ref{eq-nh-5}). 

Observe that this method preserves exactly the nonholonomic constraints, even though the method is based originally on the mid-point discretization.

\section{Future work} \label{Sec:Future}
The mathematical results obtained in this paper open some interesting research lines: 
\begin{itemize}
\item The application of our techniques to optimal control problems, vakonomic dynamics and, in general, systems defined using Morse families like in \cite{cendra1}.
Adding dissipative forces is also an straightforward work using the techniques in our paper. 
\item The development of examples on nonlinear spaces (Lie groups, etc) using discretization maps on manifolds.  In our examples we have typically worked on vector spaces (specially using the midpoint rule), but it is not a restriction of our methods. 
\item The construction of geometric integrators for reduced system such as controlled Euler-Poincar\'e equations... Reduced systems are of great interest in applications. A combination of the recent results obtained in \cite{liñán2025retractionmapsseedgeometric} and the methods developed in our paper will lead to those geometric integrators.
\item The addition of holonomic constraints is a noteworthy strategy to avoid to work with non-linear spaces. The idea is to derive geometric integrators for general Dirac systems defined on submanifolds of an euclidean space adding holonomic constraints into the picture  as in \cite{liñán2024newperspectivesymplecticintegration}.

\item The method for nonholonomic systems proposed in Section  \ref{section:method 2} is new in the extensive literature on the subject (see \cite{klas}
and refereces therein). Observe that for construction the method preserves exactly the nonholonomic constraints even though it is based on the mid-point rule. 
In a forthcoming paper, we will study the energy behavior of that method and produce other methods based on different discretization maps, as well as higher-order methods based on this technique.
\item In this paper we have considered Dirac systems on the ``extended" sense, that is, almost-Dirac systems. If the Dirac structure is integrable it would be interesting to perform discretizations that preserve the structure (symplectic integrators, Poisson integrators, presymplectic integrators...). This is an open problem in the geometric integration literature (see, for instance, \cite{hairer}). In a future paper, we want to derive geometric integrators based on the structure of presymplectic groupoid which is the geometric discretization of a Dirac structure (see, for instance,  \cite{BCW2004,BIL2024}).

\end{itemize}
\section*{Acknowledgements}

The authors acknowledge financial support from Grants PID2022-137909-NB-C21,    RED2022-134301-TD and CEX2023-001347-S funded by MCIN/AEI/10.13039/501100011033.

\bibliography{references}

\end{document}